\documentclass[a4paper,reqno,10pt]{amsart}

\usepackage{graphicx}
\usepackage{amsmath,mathtools,latexsym,amssymb,amsthm}
\usepackage[nospace,noadjust]{cite}
\usepackage{stmaryrd}
\usepackage{cancel}
\usepackage[all,cmtip]{xy}
\usepackage{enumerate}
\usepackage{setspace}
\usepackage[mathscr]{euscript}
\theoremstyle{plain}
\newtheorem{thm}{Theorem}[section]
\newtheorem{prop}[thm]{Proposition}
\newtheorem{lemma}[thm]{Lemma}

\theoremstyle{definition}
\newtheorem{defn}[thm]{Definition}
\newtheorem{remark}[thm]{Remark}
\newtheorem{exam}[thm]{Example}

\newcommand{\g}{\mathfrak g}

\newcommand{\bk}{{\mathbf{k}}}
\newcommand{\ot}{\otimes}
\newcommand{\id}{\mathrm{id}}
\newcommand{\Hom}{\mathrm{Hom}}
\newcommand{\Hoch}{\mathrm{Alg}}
\newcommand{\NFa}{\mathrm{NjA}}
\newcommand{\NF}{\mathrm{NjO}}
\newcommand{\rmH}{\mathrm{H}}
\newcommand{\HH}{\mathrm{H}_{\mathrm{Alg}}}

\newcommand{\Z}{\mathbb{Z}}
\SetSymbolFont{stmry}{bold}{U}{stmry}m{n}

\usepackage{hyperref}
\hypersetup{colorlinks, citecolor=blue, filecolor=black, linkcolor=blue, urlcolor=black}

\begin{document}

\title[Nijenhuis family algebras]{Cohomology and homotopy theory of Nijenhuis family algebras}

\author{Sami Benabdelhafidh}
\address{University of Sfax, Faculty of Sciences of Sfax, BP 1171, 3038 Sfax, Tunisia}
\renewcommand{\emailaddrname}{Email address}
\email{abdelhafidhsami41@gmail.com,
samibenabdelhafidh95@gmail.com}


\begin{abstract}
Nijenhuis family algebras generalize Nijenhuis associative algebras, the operator being replaced by a family of operators labelled by a semigroup $\Omega$. In this paper, we introduce Nijenhuis family bimodules and show that a Nijenhuis family algebra gives rise to an associative algebra on ${A\ot\bk\Omega}$. We realize Nijenhuis families and Nijenhuis family algebras as Maurer--Cartan elements of a graded Lie algebra and of an $L_\infty$-algebra. We define the cohomology of a Nijenhuis family algebra with coefficients in a bimodule via a mapping cone, obtain a long exact sequence, and show that infinitesimal deformations are classified by the second cohomology group. Finally, we show that strict homotopy Nijenhuis family algebras induce $A_\infty$-algebras on ${A\ot\bk\Omega}$.
\end{abstract}

\maketitle

\tableofcontents

\thispagestyle{empty}
\section{Introduction}

Nijenhuis operators arise naturally in the study of integrable systems and have attracted considerable attention in differential geometry and algebra. Originating in the work of Nijenhuis \cite{Nij51}, they play an important role in the theory of Poisson--Nijenhuis structures \cite{KSM90}. In the associative setting, a linear map $N:A\to A$ on an associative algebra $(A,\cdot)$ is called a Nijenhuis operator if
\begin{equation*}\label{eq:nijenhuis}
N(a)\cdot N(b)
=
N\bigl(N(a)\cdot b+a\cdot N(b)-N(a\cdot b)\bigr)
\end{equation*}
for all $a,b\in A$. Equivalently, the operation
\[
a\cdot_N b=N(a)\cdot b+a\cdot N(b)-N(a\cdot b)
\]
defines an associative product on $A$, and $N$ is a morphism from $(A,\cdot_N)$ to $(A,\cdot)$. Nijenhuis operators are closely related to Rota--Baxter operators and other algebraic structures \cite{EbrahimiFard2004,FBP}. Their connections with dendriform, NS and N-dendriform algebras have also been investigated in \cite{Leroux,LeiGuo}. 
More recently,
Nijenhuis operators have been extensively investigated in
various algebraic settings, including Lie algebras
\cite{NijenhuisRichardson1967}, Leibniz algebras
\cite{MondalSaha2024}, pre-Lie algebras
\cite{WangShengBaiLiu2019}, and mock-Lie bialgebras
\cite{Mabrouk}.

Deformation theory provides a natural framework for investigating algebraic structures and their operators. The deformation theory of associative algebras is based on the Hochschild cohomology complex \cite{hoch,gers}, while the deformation theory of algebra morphisms was developed in \cite{gers-sch}. These approaches led to the systematic use of differential graded Lie algebras, Maurer--Cartan equations and $L_\infty$-algebras. Their homotopy-theoretic foundations were developed in \cite{Stasheff,LM}. In the context of Nijenhuis algebras, deformation theory and minimal models have recently been investigated in \cite{song,Ben}.

The study of algebraic structures equipped with families of compatible operations has led to the development of multiple-operator algebras \cite{Kur}. In particular, Rota--Baxter family algebras and their associated dendriform family structures have been extensively studied \cite{Guo09,ZG,ZGM,ZGM23}. Related developments include pre-Lie family algebras \cite{ZM}, matching Rota--Baxter algebras \cite{ZGG,GGZ21}, generalized dendriform and Rota--Baxter structures \cite{Foi18,Foi20,FP}, and Rota--Baxter family conformal algebras \cite{ZZL}. These constructions illustrate the importance of semigroup-indexed families of operators in the study of algebraic structures and their compatibility relations.

Motivated by these developments, the notion of a Nijenhuis family algebra was introduced in \cite{Das}. Let $A$ be an associative algebra and let $\Omega$ be a semigroup. A family of linear maps $N_\alpha:A\to A$, indexed by $\alpha\in\Omega$, is called a Nijenhuis family if
\begin{equation}\label{eq:NF-intro}
N_\alpha(a)\cdot N_\beta(b)
=
N_{\alpha\beta}\bigl(
N_\alpha(a)\cdot b
+a\cdot N_\beta(b)
-N_{\alpha\beta}(a\cdot b)
\bigr)
\end{equation}
for all $a,b\in A$ and $\alpha,\beta\in\Omega$. When $\Omega$ consists of a single element, this identity reduces to the defining identity of an ordinary associative Nijenhuis operator. Thus, Nijenhuis family algebras extend the classical notion by incorporating compatibility relations among a collection of operators indexed by a semigroup. The relationship between twisted Rota--Baxter families and NS-family algebras was also established in \cite{Das}.

Cohomology and deformation theories for families of operators have recently received increasing attention. The cohomology and deformation theory of $\Omega$-Rota--Baxter algebras of arbitrary weight were developed in \cite{Wang}, while related constructions for twisted Rota--Baxter family operators on Hom-associative algebras were studied in \cite{Teng}. More recently, the deformation and homotopy theory of Rota--Baxter family algebras was investigated in \cite{Ddas}. These works motivate the development of a corresponding cohomological and homotopical framework for Nijenhuis family algebras.

\medskip

The aim of the present paper is to develop the cohomology, deformation and homotopy theory of Nijenhuis family algebras $(A,\{N_\omega\}_{\omega\in\Omega})$, with coefficients in Nijenhuis family bimodules $(M,\{N_{M,\omega}\}_{\omega\in\Omega})$. First, we show that the products
\[
a\cdot_{N;\alpha,\beta}b
=
N_\alpha(a)\cdot b+a\cdot N_\beta(b)
-N_{\alpha\beta}(a\cdot b)
\]
induce an associative product
\[
(a\otimes\alpha)\ast(b\otimes\beta)
=
(a\cdot_{N;\alpha,\beta}b)\otimes\alpha\beta
\]
on $A\otimes\bk\Omega$, defining an associative algebra $A_N$ (Proposition~\ref{prop:Omega-assoc}). We also show that every Nijenhuis family bimodule is naturally an $A_N$-bimodule (Proposition~\ref{prop:MP}). Furthermore, we construct a graded Lie algebra whose Maurer--Cartan elements are precisely the Nijenhuis families on $A$ (Theorem~\ref{thm:MC-family}), together with an $L_\infty$-algebra governing Nijenhuis family algebra structures in which both the associative product and the operators are allowed to vary (Theorem~\ref{thm:MC-NFalg}).

Next, we define the cohomology $\rmH^\bullet_\NFa(A,M)$ as the cohomology of the mapping cone of a chain map $\Phi$ from the Hochschild complex to the cochain complex associated with the family of operators (Proposition~\ref{prop:Phi-chain} and Definition~\ref{def:cone}). This construction yields a long exact sequence (Theorem~\ref{thm:les}). Unlike the cohomology of the associated Nijenhuis associative algebra $(A\otimes\bk\Omega,\widetilde N)$, our complex distinguishes between ordinary multilinear deformations of the associative product and $\Omega$-indexed deformations of the operators. We establish an embedding of our complex into the corresponding complex of $(A\otimes\bk\Omega,\widetilde N)$ (Proposition~\ref{prop:song-embedding}). As an application, we prove that infinitesimal deformations are classified by the second cohomology group $\rmH^2_\NFa(A)$ (Theorem~\ref{thm:infinitesimal}).

Finally, using the minimal model constructed in \cite{song}, we introduce homotopy Nijenhuis family algebras and prove that their homology carries a graded Nijenhuis family algebra structure (Theorem~\ref{thm:homotopy-low}). We also show that strict homotopy Nijenhuis family algebras induce $A_\infty$-algebra structures on $A\otimes\bk\Omega$, whose degree-zero product is precisely the associative product of $A_N$ (Theorem~\ref{thm:strict-Ainfty}).

\medskip

The paper is organized as follows. In Section~\ref{sec:background}, we recall Nijenhuis family algebras, introduce Nijenhuis family bimodules and the associative algebra $A_N$, and recall $L_\infty$-algebras. In Section~\ref{sec:MC}, we realize Nijenhuis families and Nijenhuis family algebras as Maurer--Cartan elements of a graded Lie algebra and of an $L_\infty$-algebra. In Section~\ref{sec:cohomology}, we define the cohomology of a Nijenhuis family algebra with coefficients in a bimodule via a mapping cone. In Section~\ref{sec:deformations}, we show that infinitesimal deformations are classified by the second cohomology group. Finally, in Section~\ref{sec:homotopy}, we study homotopy Nijenhuis family algebras and the $A_\infty$-algebras they induce on $A\otimes\bk\Omega$.

\medskip

Throughout the paper, all vector spaces, linear maps and tensor products are over a field $\bk$ of characteristic $0$.
The semigroup $\Omega$ is not necessarily a monoid; a unit is used only in Proposition~\ref{prop:song-embedding}, where this assumption is explicitly stated.

\section{Background on Nijenhuis family algebras and \texorpdfstring{$L_\infty$}{L-infinity}-algebras}\label{sec:background}

In this section, we recall Nijenhuis family algebras, we introduce Nijenhuis family bimodules, and we show that a Nijenhuis family algebra gives rise to a new associative algebra. We also recall some background on $L_\infty$-algebras. Our main references are \cite{Das,Ddas,song}.

\subsection{Nijenhuis family algebras and Nijenhuis family bimodules}

\begin{defn}[\cite{Das}]\label{def:NF-algebra}
Let $(A,\mu=\cdot)$ be an associative algebra. A collection $\{N_\omega:A\to A\}_{\omega\in\Omega}$ of linear operators labelled by the elements of $\Omega$ is a \emph{Nijenhuis family} if it satisfies \eqref{eq:NF-intro}, that is,
\[N_\alpha(a)\cdot N_\beta(b)=N_{\alpha\beta}\big(N_\alpha(a)\cdot b+a\cdot N_\beta(b)-N_{\alpha\beta}(a\cdot b)\big)\]
for any $a,b\in A$ and $\alpha,\beta\in\Omega$, or in terms of maps
\begin{equation}\label{eq:NF-maps}
\mu\circ(N_\alpha\ot N_\beta)=N_{\alpha\beta}\circ\big(\mu\circ(N_\alpha\ot\id)+\mu\circ(\id\ot N_\beta)-N_{\alpha\beta}\circ\mu\big).
\end{equation}
In this case, $A=(A,\mu,\{N_\omega\}_{\omega\in\Omega})$ is called a \emph{Nijenhuis family algebra}.

Let $B=(B,\mu_B,\{N_{B,\omega}\})$ and $A=(A,\mu_A,\{N_{A,\omega}\})$ be two Nijenhuis family algebras. A morphism of associative algebras $f:B\to A$ is called a \emph{morphism of Nijenhuis family algebras} if $N_{A,\omega}\circ f=f\circ N_{B,\omega}$ for all $\omega\in\Omega$.
\end{defn}

We give some examples of Nijenhuis families.

\begin{exam}\label{ex:NF}
Let $(A,\cdot)$ be an associative algebra, and let $a,b\in A$.
\begin{enumerate}[(i)]
\item Any Nijenhuis operator $N$ gives the constant Nijenhuis family $N_\omega=N$; for instance, for $x\in A$, the left multiplication $N(c)=x\cdot c$ is a Nijenhuis operator, since
\[N\big(x\cdot a\cdot b+a\cdot x\cdot b-x\cdot a\cdot b\big)=x\cdot a\cdot x\cdot b=N(a)\cdot N(b),\]
and so is the right multiplication. If $\{N_\omega\}$ is a Nijenhuis family and $\lambda\in\bk$, so are $\{\lambda N_\omega\}$ and $\{N_\omega+\lambda\,\id\}$; for the latter, by \eqref{eq:NF-intro}, $(N_{\alpha\beta}+\lambda)\big(N_\alpha(a)b+aN_\beta(b)-N_{\alpha\beta}(ab)+\lambda ab\big)=N_\alpha(a)N_\beta(b)+\lambda N_\alpha(a)b+\lambda aN_\beta(b)+\lambda^2ab=(N_\alpha+\lambda)(a)\,(N_\beta+\lambda)(b)$.
\item Let $\{x_\omega\}_{\omega\in\Omega}\subseteq A$. If $\Omega$ is a right zero semigroup ($\alpha\beta=\beta$), then $N_\omega(a):=a\cdot x_\omega$ is a Nijenhuis family, since $N_{\alpha\beta}=N_\beta$ and
\[\big(a\cdot x_\alpha\cdot b+a\cdot b\cdot x_\beta-a\cdot b\cdot x_\beta\big)\cdot x_\beta=N_\alpha(a)\cdot N_\beta(b).\]
Similarly, if $\Omega$ is a left zero semigroup ($\alpha\beta=\alpha$), then $N_\omega(a):=x_\omega\cdot a$ is a Nijenhuis family, since $N_{\alpha\beta}=N_\alpha$ and
\[x_\alpha\cdot\big(x_\alpha\cdot a\cdot b+a\cdot x_\beta\cdot b-x_\alpha\cdot a\cdot b\big)=x_\alpha\cdot a\cdot x_\beta\cdot b=N_\alpha(a)\cdot N_\beta(b).\]

\end{enumerate}
\end{exam}

Let $A$ be an associative algebra. Recall that an \emph{$A$-bimodule} is a vector space $M$ with two linear maps $A\ot M\to M$, $a\ot m\mapsto a\cdot m$, and $M\ot A\to M$, $m\ot a\mapsto m\cdot a$, such that
\[(a\cdot b)\cdot m=a\cdot(b\cdot m),\qquad (a\cdot m)\cdot b=a\cdot(m\cdot b),\qquad (m\cdot a)\cdot b=m\cdot(a\cdot b),\qquad \forall\, a,b\in A,\ m\in M.\]
Equivalently, $A\oplus M$ with $(a,m)\cdot(b,n)=(a\cdot b,a\cdot n+m\cdot b)$ is an associative algebra, the \emph{semi-direct product} $A\ltimes M$. Any associative algebra $A$ is an $A$-bimodule, called the \emph{adjoint} (or regular) bimodule.

\begin{defn}\label{def:NF-bimodule}
Let $A=(A,\mu,\{N_\omega\}_{\omega\in\Omega})$ be a Nijenhuis family algebra and $M$ be a bimodule over the associative algebra $(A,\mu)$. Then $M$ is said to be a \emph{Nijenhuis family bimodule} over the Nijenhuis family algebra $A$ if $M$ is endowed with a collection $\{N_{M,\,\omega}:M\to M\}_{\omega\in\Omega}$ of linear operators such that the following equations
\begin{align}
\label{eq:left}N_\alpha(a)\cdot N_{M,\,\beta}(m)&=N_{M,\,\alpha\beta}\big(N_\alpha(a) \cdot m+a \cdot N_{M,\,\beta}(m)-N_{M,\,\alpha\beta}( a \cdot m)\big),\\
\label{eq:right}N_{M,\,\alpha}(m) \cdot N_{\beta}(a)&=N_{M,\,\alpha\beta}\big(N_{M,\,\alpha}(m) \cdot a+m \cdot N_\beta(a)-N_{M,\,\alpha\beta} (m \cdot a)\big).
\end{align}
hold for any $a\in A$, $m\in M$ and $\alpha,\beta\in\Omega$. We write $M=(M,\{N_{M,\,\omega}\}_{\omega\in\Omega})$.
\end{defn}

Obviously, a Nijenhuis family algebra $A=(A,\mu,\{N_\omega\})$ is a Nijenhuis family bimodule over itself, called the \emph{regular Nijenhuis family bimodule}, since \eqref{eq:left} and \eqref{eq:right} then reduce to \eqref{eq:NF-intro}. As in the case of a single operator \cite[Proposition~1.3]{song}, Nijenhuis family bimodules are characterized by semi-direct products.

\begin{prop}\label{prop:semidirect}
Let $A=(A,\mu,\{N_\omega\}_{\omega\in\Omega})$ be a Nijenhuis family algebra, $M$ a bimodule over the associative algebra $(A,\mu)$ and $\{N_{M,\,\omega}\}_{\omega\in\Omega}$ a collection of linear operators on $M$. Put $\widehat N_\omega(a,m):=(N_\omega(a),N_{M,\,\omega}(m))$. Then $(A\ltimes M,\{\widehat N_\omega\})$ is a Nijenhuis family algebra such that the inclusion $\iota:A\to A\ltimes M$ and the projection $\pi:A\ltimes M\to A$ are morphisms of Nijenhuis family algebras if and only if $M=(M,\{N_{M,\,\omega}\})$ is a Nijenhuis family bimodule over $A$. This Nijenhuis family algebra is still denoted by $A\ltimes M$ and called the \emph{semi-direct product} of $A$ by $M$.
\end{prop}

\begin{proof}
Let $x=(a,m)$, $y=(b,n)\in A\ltimes M$ and $\alpha,\beta\in\Omega$. Since $m\cdot n=0$ in $A\ltimes M$, we have
\[\widehat N_\alpha(x)\cdot\widehat N_\beta(y)=\Big(N_\alpha(a)\cdot N_\beta(b),\ N_\alpha(a)\cdot N_{M,\,\beta}(n)+N_{M,\,\alpha}(m)\cdot N_\beta(b)\Big).\]
On the other hand, $\widehat N_{\alpha\beta}\big(\widehat N_\alpha(x)\cdot y+x\cdot\widehat N_\beta(y)-\widehat N_{\alpha\beta}(x\cdot y)\big)$ has $A$-component $N_{\alpha\beta}(a\cdot_{N;\alpha,\,\beta}b)$ and $M$-component
\[N_{M,\,\alpha\beta}\big(N_\alpha(a)\cdot n+a\cdot N_{M,\,\beta}(n)-N_{M,\,\alpha\beta}(a\cdot n)\big)+N_{M,\,\alpha\beta}\big(N_{M,\,\alpha}(m)\cdot b+m\cdot N_\beta(b)-N_{M,\,\alpha\beta}(m\cdot b)\big).\]
For $a\in A$, $m\in M$, denote by $E_l(a,m)$ (resp. $E_r(m,a)$) the difference of the left and right hand sides of \eqref{eq:left} (resp. \eqref{eq:right}). By \eqref{eq:Nstar}, the $A$-components of the two expressions above agree, so
\[\widehat N_\alpha(x)\cdot\widehat N_\beta(y)-\widehat N_{\alpha\beta}\big(\widehat N_\alpha(x)\cdot y+x\cdot\widehat N_\beta(y)-\widehat N_{\alpha\beta}(x\cdot y)\big)=\big(0,\ E_l(a,n)+E_r(m,b)\big).\]
Hence $\{\widehat N_\omega\}$ is a Nijenhuis family on $A\ltimes M$ if and only if $E_l(a,n)+E_r(m,b)=0$ for all $a,b\in A$, $m,n\in M$, $\alpha,\beta\in\Omega$. Taking $m=0$ gives \eqref{eq:left}, taking $n=0$ gives \eqref{eq:right}, and conversely \eqref{eq:left} and \eqref{eq:right} give $E_l=E_r=0$. Finally, the inclusion $\iota(a)=(a,0)$ and the projection $\pi(a,m)=a$ are algebra morphisms which satisfy $\widehat N_\omega\circ\iota=\iota\circ N_\omega$ and $\pi\circ\widehat N_\omega=N_\omega\circ\pi$.
\end{proof}

\begin{exam}\label{ex:integral}
Let $\Omega=(\mathbb R,+)$ be a semigroup, and let $R$ be the $\mathbb R$-algebra consisting of all continuous functions from $\mathbb R$ to $\mathbb R$, with $(fg)(x):=f(x)g(x)$. Let $A:=R[\varepsilon]/(\varepsilon^2)=\{f+\varepsilon g\mid f,g\in R\}$, the multiplication being defined by
\[(f_1+\varepsilon g_1)(f_2+\varepsilon g_2):=f_1f_2+\varepsilon(f_1g_2+g_1f_2).\]
Fix $\theta\in R$ and, for any $\alpha\in\Omega$, define a linear operator $N_\alpha:A\to A$ by
\[N_\alpha(f+\varepsilon g)(x):=e^{-\alpha\theta(x)}\int_0^xe^{\alpha\theta(t)}g(t)\,dt.\]
Then $\{N_\alpha\}_{\alpha\in\Omega}$ is a Nijenhuis family on $A$. In fact, let $a=f_1+\varepsilon g_1$ and $b=f_2+\varepsilon g_2$ in $A$. Since $N_\alpha(a),N_\beta(b)\in R$ and $N_{\alpha+\beta}(ab)\in R$, the coefficient of $\varepsilon$ in $N_\alpha(a)b+aN_\beta(b)-N_{\alpha+\beta}(ab)$ is $N_\alpha(a)g_2+g_1N_\beta(b)$, and $N_{\alpha+\beta}$ only depends on this coefficient. Hence
\begin{align*}
&N_{\alpha+\beta}\big(N_\alpha(a)b+aN_\beta(b)-N_{\alpha+\beta}(ab)\big)(x)\\
&=e^{-(\alpha+\beta)\theta(x)}\int_0^xe^{(\alpha+\beta)\theta(t)}\big(N_\alpha(a)(t)g_2(t)+g_1(t)N_\beta(b)(t)\big)\,dt\\
&=e^{-(\alpha+\beta)\theta(x)}\int_0^xe^{(\alpha+\beta)\theta(t)}\Big(g_2(t)e^{-\alpha\theta(t)}\int_0^te^{\alpha\theta(s)}g_1(s)\,ds+g_1(t)e^{-\beta\theta(t)}\int_0^te^{\beta\theta(s)}g_2(s)\,ds\Big)dt\\
&=e^{-(\alpha+\beta)\theta(x)}\Big(\int_0^xe^{\beta\theta(t)}g_2(t)\Big(\int_0^te^{\alpha\theta(s)}g_1(s)\,ds\Big)dt+\int_0^xe^{\alpha\theta(t)}g_1(t)\Big(\int_0^te^{\beta\theta(s)}g_2(s)\,ds\Big)dt\Big)\\
&=e^{-(\alpha+\beta)\theta(x)}\Big(\iint_{0\leq s\leq t\leq x}e^{\alpha\theta(s)}g_1(s)e^{\beta\theta(t)}g_2(t)\,ds\,dt+\iint_{0\leq t\leq s\leq x}e^{\alpha\theta(s)}g_1(s)e^{\beta\theta(t)}g_2(t)\,ds\,dt\Big)\\
&=e^{-\alpha\theta(x)}\int_0^xe^{\alpha\theta(s)}g_1(s)\,ds\cdot e^{-\beta\theta(x)}\int_0^xe^{\beta\theta(t)}g_2(t)\,dt=N_\alpha(a)(x)\,N_\beta(b)(x),
\end{align*}
where, in the second double integral, we exchanged the names of the variables $s$ and $t$; here $x\geq0$, and for $x<0$ the same computation holds with the triangles $\{x\leq t\leq s\leq0\}$ and $\{x\leq s\leq t\leq0\}$. Hence $N_{\alpha+\beta}\big(N_\alpha(a)b+aN_\beta(b)-N_{\alpha+\beta}(ab)\big)=N_\alpha(a)N_\beta(b)$. Note that $A$ is the semi-direct product $R\ltimes R$, that $N_\alpha\circ N_\beta=0$, and that this example is a particular case of \cite[Proposition~2.11]{Das}. By Example~\ref{ex:NF}(i), for $\lambda\neq0$, $\{N_\alpha+\lambda\,\id\}_{\alpha\in\Omega}$ is a Nijenhuis family on $A$ which is neither constant (when $\theta$ is not constant) nor nilpotent.
\end{exam}

The following observation relates Nijenhuis families on $A$ to Nijenhuis operators on $A\ot\bk\Omega$.

\begin{lemma}\label{lem:Ntilde}
Let $(A,\{N_\omega\}_{\omega\in\Omega})$ be a Nijenhuis family algebra. Then $\widetilde N(a\ot\omega):=N_\omega(a)\ot\omega$ is a Nijenhuis operator on the associative algebra $A\ot\bk\Omega$, with product $(a\ot\alpha)(b\ot\beta)=(a\cdot b)\ot\alpha\beta$.
\end{lemma}

\begin{proof}
Let $X=x\ot\alpha$ and $Y=y\ot\beta$. Then
\begin{align*}
\widetilde N(X)\,\widetilde N(Y)&=\big(N_\alpha(x)\cdot N_\beta(y)\big)\ot\alpha\beta,\\
\widetilde N\big(\widetilde N(X)Y+X\widetilde N(Y)-\widetilde N(XY)\big)&=N_{\alpha\beta}\big(N_\alpha(x)\cdot y+x\cdot N_\beta(y)-N_{\alpha\beta}(x\cdot y)\big)\ot\alpha\beta,
\end{align*}
which agree by \eqref{eq:NF-intro}.
\end{proof}

To define the cohomology theories of this paper, we use cochains labelled by the elements of $\Omega$. Let $A$ be an associative algebra and $M$ an $A$-bimodule. For $n\geq0$, let $\Hom_\Omega(A^{\ot n},M)$ be the space of collections $f=\{f_{\alpha_1,\ldots,\alpha_n}:A^{\ot n}\to M\}_{\alpha_1,\ldots,\alpha_n\in\Omega}$ of multilinear maps labelled by the elements of $\Omega^{\times n}$ \cite[Section~3.1]{Ddas}. For $n=0$, we put $\Hom_\Omega(A^{\ot 0},M):=M$. Any $f\in\Hom(A^{\ot n},M)$ is regarded as the constant collection $f_{\alpha_1,\ldots,\alpha_n}=f$. Define $\delta_{\Hoch,M}:\Hom_\Omega(A^{\ot n},M)\to\Hom_\Omega(A^{\ot n+1},M)$ by
\begin{align*}
\big(\delta^n_{\Hoch,M}(f)\big)_{\alpha_1, \ldots, \alpha_{n+1}} (a_1, \ldots, a_{n+1} ) ={}& a_1 \cdot f_{\alpha_2, \ldots, \alpha_{n+1}} (a_2, \ldots, a_{n+1})  \\
&+ \sum_{i=1}^n (-1)^{i}  f_{\alpha_1, \ldots, \alpha_i \alpha_{i+1}, \ldots, \alpha_{n+1}} \big(  a_1, \ldots, a_i \cdot a_{i+1}, \ldots, a_{n+1}  \big)  \\
&+(-1)^{n+1}  f_{\alpha_1, \ldots, \alpha_n} (a_1, \ldots, a_n) \cdot a_{n+1}.
\end{align*}
When $f$ does not depend on the indices, $\delta^n_{\Hoch,M}(f)$ is the usual Hochschild coboundary of $f\in\Hom(A^{\ot n},M)$. When $M=A$ is the regular bimodule, we simply write $\delta^n_\Hoch$.

Let $(A,\{N_\omega\}_{\omega\in\Omega})$ be a Nijenhuis family algebra. For $a,b\in A$ and $\alpha,\beta\in\Omega$, put
\[a\cdot_{N;\alpha ,\,\beta} b:=N_\alpha(a)\cdot b+a\cdot N_\beta(b)-N_{\alpha \beta}( a\cdot b).\]
Then \eqref{eq:NF-intro}
reads
\begin{equation}\label{eq:Nstar}
N_{\alpha\beta}(a\cdot_{N;\alpha,\,\beta}b)=N_\alpha(a)\cdot N_\beta(b),\qquad \forall\, a,b\in A,\ \alpha,\beta\in\Omega.
\end{equation}
As for a single Nijenhuis operator \cite[Proposition~1.4]{song}, the products $\cdot_{N;\alpha,\beta}$ give rise to a new associative algebra, as follows.

\begin{prop}\label{prop:Omega-assoc}
Let $A=(A,\mu,\{N_\omega\}_{\omega\in\Omega})$ be a Nijenhuis family algebra, and let $A\ot\bk\Omega$ be the associative algebra with product $(a\ot\alpha)(b\ot\beta)=(a\cdot b)\ot\alpha\beta$. Define a new binary operation on $A\ot\bk\Omega$ by
\[(a\ot\alpha)\ast(b\ot\beta):=(a\cdot_{N;\alpha,\,\beta}b)\ot\alpha\beta,\qquad a\cdot_{N;\alpha,\,\beta}b=N_\alpha(a)\cdot b+a\cdot N_\beta(b)-N_{\alpha\beta}(a\cdot b).\]
Then:
\begin{enumerate}
\item the operation $\ast$ is associative, so $A_N:=(A\ot\bk\Omega,\ast)$ is a new associative algebra; equivalently, for all $a,b,c\in A$ and $\alpha,\beta,\gamma\in\Omega$,
\begin{equation}\label{eq:star-assoc}
(a\cdot_{N;\alpha,\,\beta}b)\cdot_{N;\alpha\beta,\,\gamma}c=a\cdot_{N;\alpha,\,\beta\gamma}(b\cdot_{N;\beta,\,\gamma}c);
\end{equation}
\item the map $\widetilde N:A_N\to A\ot\bk\Omega$, $\widetilde N(a\ot\omega)=N_\omega(a)\ot\omega$, is a morphism of associative algebras, that is, $N_{\alpha\beta}(a\cdot_{N;\alpha,\beta}b)=N_\alpha(a)\cdot N_\beta(b)$;
\item $\widetilde N$ is a Nijenhuis operator on the associative algebra $A_N$, that is, for all $a,b\in A$ and $\alpha,\beta\in\Omega$,
\[N_\alpha(a)\cdot_{N;\alpha,\beta}N_\beta(b)=N_{\alpha\beta}\big(N_\alpha(a)\cdot_{N;\alpha,\beta}b+a\cdot_{N;\alpha,\beta}N_\beta(b)-N_{\alpha\beta}(a\cdot_{N;\alpha,\beta}b)\big).\]
\end{enumerate}
\end{prop}

\begin{proof}
(1) By Lemma~\ref{lem:Ntilde}, $\widetilde N$ is a Nijenhuis operator on the associative algebra $A\ot\bk\Omega$, and $\ast$ is its deformed product $X\ast Y=\widetilde N(X)Y+X\widetilde N(Y)-\widetilde N(XY)$; hence $\ast$ is associative \cite[Proposition~1.4]{song}. Explicitly, for $a,b,c\in A$ and $\alpha,\beta,\gamma\in\Omega$,
\begin{align*}
\big((a\ot\alpha)\ast(b\ot\beta)\big)\ast(c\ot\gamma)&=\big((a\cdot_{N;\alpha,\,\beta}b)\cdot_{N;\alpha\beta,\,\gamma}c\big)\ot\alpha\beta\gamma,\\
(a\ot\alpha)\ast\big((b\ot\beta)\ast(c\ot\gamma)\big)&=\big(a\cdot_{N;\alpha,\,\beta\gamma}(b\cdot_{N;\beta,\,\gamma}c)\big)\ot\alpha\beta\gamma,
\end{align*}
and, since $x\ot\alpha\beta\gamma=0$ implies $x=0$, the associativity of $\ast$ is equivalent to \eqref{eq:star-assoc}.

(2) We have $\widetilde N\big((a\ot\alpha)\ast(b\ot\beta)\big)=N_{\alpha\beta}(a\cdot_{N;\alpha,\beta}b)\ot\alpha\beta$ and $\widetilde N(a\ot\alpha)\,\widetilde N(b\ot\beta)=\big(N_\alpha(a)\cdot N_\beta(b)\big)\ot\alpha\beta$; they agree by \eqref{eq:Nstar}.

(3) Applying \eqref{eq:Nstar} to $(N_\alpha(a),b)$, to $(a,N_\beta(b))$ and to $(a,b)$, we get
\begin{align*}
N_\alpha(a)\cdot_{N;\alpha,\beta}N_\beta(b)&=N_\alpha\big(N_\alpha(a)\big)\cdot N_\beta(b)+N_\alpha(a)\cdot N_\beta\big(N_\beta(b)\big)-N_{\alpha\beta}\big(N_\alpha(a)\cdot N_\beta(b)\big)\\
&=N_{\alpha\beta}\big(N_\alpha(a)\cdot_{N;\alpha,\beta}b\big)+N_{\alpha\beta}\big(a\cdot_{N;\alpha,\beta}N_\beta(b)\big)-N_{\alpha\beta}\big(N_{\alpha\beta}(a\cdot_{N;\alpha,\beta}b)\big).
\end{align*}
Tensoring with $\alpha\beta$, this is
\[\widetilde N(X)\ast\widetilde N(Y)=\widetilde N\big(\widetilde N(X)\ast Y+X\ast\widetilde N(Y)-\widetilde N(X\ast Y)\big)\]
for $X=a\ot\alpha$ and $Y=b\ot\beta$.
\end{proof}

Let now $(M,\{N_{M,\,\omega}\})$ be a Nijenhuis family bimodule over $(A,\{N_\omega\})$. We denote by $\partial$ the operator on $\bigoplus_{n\geq1}\Hom_\Omega(A^{\ot n},M)$ given by the formula of $\delta_{\Hoch,M}$ in which the product of $A$ is replaced by $\cdot_{N;\alpha,\,\beta}$ and the actions by $(a,m)\mapsto N_\alpha(a)\cdot m$ and $(m,a)\mapsto m\cdot N_\beta(a)$:
\begin{align*}
\big(\partial (f) \big)_{\alpha_1, \ldots, \alpha_{n+1}} (a_1, \ldots, a_{n+1} ) ={}& N_{\alpha_1}(a_1) \cdot f_{\alpha_2, \ldots, \alpha_{n+1}} (a_2, \ldots, a_{n+1})  \\
&+ \sum_{i=1}^n (-1)^{i}  f_{\alpha_1, \ldots, \alpha_i \alpha_{i+1}, \ldots, \alpha_{n+1}} \big(  a_1, \ldots, a_i \cdot_{N;\alpha_i,\,\alpha_{i+1}} a_{i+1}, \ldots, a_{n+1}  \big)  \\
&+(-1)^{n+1}  f_{\alpha_1, \ldots, \alpha_n} (a_1, \ldots, a_n) \cdot N_{\alpha_{n+1}}(a_{n+1}).
\end{align*}

One can also obtain new bimodules from Nijenhuis family bimodules.

\begin{prop}\label{prop:MP}
Let $A=(A,\mu,\{N_\omega\})$ be a Nijenhuis family algebra and $M=(M,\{N_{M,\,\omega}\})$ a Nijenhuis family bimodule over it. Define a left action $\lhd$ and a right action $\rhd$ of $A_N$ on $M$ by
\[(a\ot\alpha)\lhd m:=N_\alpha(a)\cdot m,\qquad m\rhd(a\ot\alpha):=m\cdot N_\alpha(a),\qquad a\in A,\ m\in M,\ \alpha\in\Omega.\]
Then:
\begin{enumerate}
\item these actions make $M$ into a bimodule over the associative algebra $A_N$, denoted by $M_N$;
\item for $f\in\Hom_\Omega(A^{\ot n},M)$, let $\widetilde f\in\Hom\big((A\ot\bk\Omega)^{\ot n},M\big)$ be given by $\widetilde f(a_1\ot\alpha_1,\ldots,a_n\ot\alpha_n):=f_{\alpha_1,\ldots,\alpha_n}(a_1,\ldots,a_n)$. Then $f\mapsto\widetilde f$ is an isomorphism, and $\widetilde{\partial(f)}$ is the Hochschild coboundary of $\widetilde f$ in the cochain complex of the associative algebra $A_N$ with coefficients in $M_N$.
\end{enumerate}
\end{prop}

\begin{proof}
(1) Let $a,b\in A$, $m\in M$ and $\alpha,\beta\in\Omega$. By \eqref{eq:Nstar},
\begin{align*}
\big((a\ot\alpha)\ast(b\ot\beta)\big)\lhd m&=N_{\alpha\beta}(a\cdot_{N;\alpha,\beta}b)\cdot m=\big(N_\alpha(a)\cdot N_\beta(b)\big)\cdot m\\&=N_\alpha(a)\cdot\big(N_\beta(b)\cdot m\big)=(a\ot\alpha)\lhd\big((b\ot\beta)\lhd m\big),\\
\big((a\ot\alpha)\lhd m\big)\rhd(b\ot\beta)&=\big(N_\alpha(a)\cdot m\big)\cdot N_\beta(b)=N_\alpha(a)\cdot\big(m\cdot N_\beta(b)\big)=(a\ot\alpha)\lhd\big(m\rhd(b\ot\beta)\big),\\
\big(m\rhd(a\ot\alpha)\big)\rhd(b\ot\beta)&=\big(m\cdot N_\alpha(a)\big)\cdot N_\beta(b)=m\cdot N_{\alpha\beta}(a\cdot_{N;\alpha,\beta}b)=m\rhd\big((a\ot\alpha)\ast(b\ot\beta)\big).
\end{align*}

(2) Since $\Omega$ is a basis of $\bk\Omega$, a linear map on $(A\ot\bk\Omega)^{\ot n}$ is the same as a collection labelled by $\Omega^{\times n}$, so $f\mapsto\widetilde f$ is an isomorphism. On $X_i=a_i\ot\alpha_i$, we have
\begin{align*}
X_1\lhd\widetilde f(X_2,\ldots,X_{n+1})&=N_{\alpha_1}(a_1)\cdot f_{\alpha_2,\ldots,\alpha_{n+1}}(a_2,\ldots,a_{n+1}),\\
\widetilde f(\ldots,X_i\ast X_{i+1},\ldots)&=f_{\alpha_1,\ldots,\alpha_i\alpha_{i+1},\ldots,\alpha_{n+1}}(\ldots,a_i\cdot_{N;\alpha_i,\alpha_{i+1}}a_{i+1},\ldots),\\
\widetilde f(X_1,\ldots,X_n)\rhd X_{n+1}&=f_{\alpha_1,\ldots,\alpha_n}(a_1,\ldots,a_n)\cdot N_{\alpha_{n+1}}(a_{n+1}),
\end{align*}
which are the three kinds of terms of $\partial(f)$.
\end{proof}

\subsection{\texorpdfstring{$L_\infty$}{L-infinity}-algebras}
We recall some basic notions on $L_\infty$-algebras \cite{LM}, following the conventions of \cite{Ddas}.

\begin{defn}
An \emph{$L_\infty$-algebra} is a graded vector space $L=\bigoplus_{i\in\Z}L_i$ with a collection $\{l_k:L^{\ot k}\to L\}_{k\geq1}$ of degree $1$ linear maps which are graded symmetric, that is, $l_k(x_{\sigma(1)},\ldots,x_{\sigma(k)})=\epsilon(\sigma)\,l_k(x_1,\ldots,x_k)$ for $\sigma\in S_k$, and satisfy, for all $n\geq1$ and homogeneous $x_1,\ldots,x_n\in L$,
\[\sum_{i+j=n+1}\ \sum_{\sigma\in S_{(i,n-i)}}\epsilon(\sigma)\,l_j\big(l_i(x_{\sigma(1)},\ldots,x_{\sigma(i)}),x_{\sigma(i+1)},\ldots,x_{\sigma(n)}\big)=0,\]
where $S_{(i,n-i)}$ is the set of $(i,n-i)$-shuffles and $\epsilon(\sigma)$ is the Koszul sign.
\end{defn}

Any graded Lie algebra $(\g,[\ ,\ ])$ gives an $L_\infty$-algebra on the desuspension $s^{-1}\g$, with $l_2(s^{-1}x,s^{-1}y)=(-1)^{|x|}s^{-1}[x,y]$ and $l_k=0$ for $k\neq2$. We assume that all $L_\infty$-algebras are weakly filtered, so that the sums below converge \cite{Ddas}.

\begin{defn}
Let $(L,\{l_k\}_{k\geq1})$ be an $L_\infty$-algebra. An element $\theta\in L_0$ is a \emph{Maurer--Cartan element} if
\[\sum_{k\geq1}\frac{1}{k!}\,l_k(\theta,\ldots,\theta)=0.\]
\end{defn}

If $\theta$ is a Maurer--Cartan element, then $l^\theta_k(x_1,\ldots,x_k):=\sum_{i\geq0}\frac1{i!}\,l_{k+i}(\theta,\ldots,\theta,x_1,\ldots,x_k)$ defines a new $L_\infty$-algebra structure on $L$, said to be obtained by \emph{twisting} by $\theta$. In particular, $(l^\theta_1)^2=0$.

\section{Maurer--Cartan characterization of Nijenhuis family algebras}\label{sec:MC}

In this section, we construct a graded Lie algebra whose Maurer--Cartan elements are Nijenhuis families, and an $L_\infty$-algebra whose Maurer--Cartan elements are Nijenhuis family algebras. The first one gives the cochain complex of a Nijenhuis family.

\subsection{A generalization of Gerstenhaber's graded Lie algebra}\label{subsec:Gerst}
Let $V$ be a vector space. Following \cite[Section~3.1]{Ddas}, the spaces $\Hom_\Omega(V^{\ot n},V)$, $n\geq1$, form a nonsymmetric operad, with partial compositions
\begin{multline*}(f\circ_i g)_{\alpha_1,\ldots,\alpha_{m+n-1}}(v_1,\ldots,v_{m+n-1}):=f_{\alpha_1,\ldots,\alpha_{i-1},\,\alpha_i\cdots\alpha_{i+n-1},\,\alpha_{i+n},\ldots,\alpha_{m+n-1}}\big(v_1,\ldots,v_{i-1},\\g_{\alpha_i,\ldots,\alpha_{i+n-1}}(v_i,\ldots,v_{i+n-1}),v_{i+n},\ldots,v_{m+n-1}\big)\end{multline*}
for $f\in\Hom_\Omega(V^{\ot m},V)$, $g\in\Hom_\Omega(V^{\ot n},V)$ and $1\leq i\leq m$, and with unit the constant family $I=\{\id_V\}_{\omega\in\Omega}$. As a consequence, with
\[f\circ g:=\sum_{i=1}^m(-1)^{(i-1)(n-1)}f\circ_i g,\qquad [f,g]_\Omega:=f\circ g-(-1)^{(m-1)(n-1)}g\circ f,\]
the graded space $\bigoplus_{n\geq1}\Hom_\Omega(V^{\ot n},V)$, with $\Hom_\Omega(V^{\ot n},V)$ in degree $n-1$, is a graded Lie algebra, called the \emph{$\Omega$-Gerstenhaber algebra} \cite[(10)]{Ddas}. The embedding of $\Hom(V^{\ot n},V)$ as the constant families is compatible with the partial compositions, so the $\Omega$-Gerstenhaber bracket restricts to the usual Gerstenhaber bracket $[\ ,\ ]_G$ on $\bigoplus_{n\geq1}\Hom(V^{\ot n},V)$.

A bilinear map $\mu:V\ot V\to V$ is associative if and only if $[\mu,\mu]_G=0$, since $\tfrac12[\mu,\mu]_G=\mu\circ_1\mu-\mu\circ_2\mu$. In that case, the bracket with $\mu$ is the Hochschild coboundary up to sign:
\begin{equation}\label{eq:delta-bracket}
[\mu,f]_\Omega=(-1)^{n-1}\delta^n_\Hoch(f),\qquad \forall\, f\in\Hom_\Omega(A^{\ot n},A),\ n\geq1,
\end{equation}
where $A=(V,\mu)$. Indeed, $\mu$ has arity $2$ and is a constant family, so
\[[\mu,f]_\Omega=\mu\circ_1f+(-1)^{n-1}\mu\circ_2f-(-1)^{n-1}\sum_{i=1}^n(-1)^{i-1}f\circ_i\mu.\]
Evaluating on $a_1,\ldots,a_{n+1}$, we have
\begin{align*}
(\mu\circ_1f)_{\alpha_1,\ldots,\alpha_{n+1}}(a_1,\ldots,a_{n+1})&=f_{\alpha_1,\ldots,\alpha_n}(a_1,\ldots,a_n)\cdot a_{n+1},\\
(\mu\circ_2f)_{\alpha_1,\ldots,\alpha_{n+1}}(a_1,\ldots,a_{n+1})&=a_1\cdot f_{\alpha_2,\ldots,\alpha_{n+1}}(a_2,\ldots,a_{n+1}),\\
(f\circ_i\mu)_{\alpha_1,\ldots,\alpha_{n+1}}(a_1,\ldots,a_{n+1})&=f_{\alpha_1,\ldots,\alpha_i\alpha_{i+1},\ldots,\alpha_{n+1}}(a_1,\ldots,a_i\cdot a_{i+1},\ldots,a_{n+1}),
\end{align*}
since the index of $f$ at the $i$-th input is the product of the indices of the inputs of $\mu$. Hence
\begin{multline*}
[\mu,f]_\Omega(a_1,\ldots,a_{n+1})=(-1)^{n-1}\Big(a_1\cdot f(a_2,\ldots,a_{n+1})+\sum_{i=1}^n(-1)^if(a_1,\ldots,a_i\cdot a_{i+1},\ldots,a_{n+1})\\+(-1)^{n+1}f(a_1,\ldots,a_n)\cdot a_{n+1}\Big)=(-1)^{n-1}\delta^n_\Hoch(f)(a_1,\ldots,a_{n+1}),
\end{multline*}
where we omitted the indices, which are those of the formula of $\delta^n_\Hoch$.

The partial compositions of collections are those of ordinary multilinear maps on $V\ot\bk\Omega$, as follows. Put $V_\Omega:=V\ot\bk\Omega$, and, for $f\in\Hom_\Omega(V^{\ot n},V)$, define $\iota(f)\in\Hom(V_\Omega^{\ot n},V_\Omega)$ by
\[\iota(f)(v_1\ot\alpha_1,\ldots,v_n\ot\alpha_n):=f_{\alpha_1,\ldots,\alpha_n}(v_1,\ldots,v_n)\ot\alpha_1\cdots\alpha_n.\]
We say that a multilinear map on $V_\Omega$ is \emph{homogeneous} if it is of the form $\iota(f)$.

\begin{lemma}\label{lem:homog}
The maps $\iota:\Hom_\Omega(V^{\ot n},V)\to\Hom(V_\Omega^{\ot n},V_\Omega)$ are injective, $\iota(I)=\id_{V_\Omega}$, and $\iota(f\circ_ig)=\iota(f)\circ_i\iota(g)$ for all $f\in\Hom_\Omega(V^{\ot m},V)$, $g\in\Hom_\Omega(V^{\ot n},V)$ and $1\leq i\leq m$. Consequently, $\iota([f,g]_\Omega)=[\iota(f),\iota(g)]_G$.
\end{lemma}

\begin{proof}
If $\iota(f)=0$, then $f_{\alpha_1,\ldots,\alpha_n}(v_1,\ldots,v_n)\ot\alpha_1\cdots\alpha_n=0$ for all arguments; since $\alpha_1\cdots\alpha_n$ is a basis element of $\bk\Omega$, $f=0$. Clearly $\iota(I)(v\ot\alpha)=v\ot\alpha$. Let $\gamma:=\alpha_i\cdots\alpha_{i+n-1}$. Evaluating $\iota(f)\circ_i\iota(g)$ on $v_1\ot\alpha_1,\ldots,v_{m+n-1}\ot\alpha_{m+n-1}$, the $i$-th input of $\iota(f)$ receives $g_{\alpha_i,\ldots,\alpha_{i+n-1}}(v_i,\ldots,v_{i+n-1})\ot\gamma$, so we obtain

\begin{multline*}f_{\alpha_1,\ldots,\alpha_{i-1},\gamma,\alpha_{i+n},\ldots}\big(\ldots,v_{i-1},g_{\alpha_i,\ldots,\alpha_{i+n-1}}(v_i,\ldots,v_{i+n-1}),v_{i+n},\ldots\big)\ot\alpha_1\cdots\alpha_{m+n-1}\\=\iota(f\circ_ig)(v_1\ot\alpha_1,\ldots),\end{multline*}
by the definition of the partial compositions of collections. The last assertion follows, since both brackets are given by the same formula in terms of partial compositions.
\end{proof}

The same computation applies verbatim when $V$ is a graded vector space and collections consist of homogeneous maps of the same degree: since $\bk\Omega$ is concentrated in degree $0$, the map $\iota$ preserves degrees and the Koszul signs on both sides of $\iota(f\circ_ig)=\iota(f)\circ_i\iota(g)$ are the same. We denote by $\mathrm{End}^\Omega_V$ the graded nonsymmetric operad formed by the spaces $\Hom_\Omega(V^{\ot n},V)$ of collections of homogeneous maps of the same degree, with these partial compositions and the Koszul sign rule \cite[Section~6.1]{Ddas}; the constant collections form the suboperad $\mathrm{End}_V$ of multilinear maps on $V$. We will use Lemma~\ref{lem:homog} repeatedly in the following form.

\begin{lemma}[Transport principle]\label{lem:transport}
Let $V$ be a graded vector space, and let $F(x_1,\ldots,x_r)$ be an expression obtained from multilinear maps $x_1,\ldots,x_r$ and the identity by iterated partial compositions and linear combinations whose coefficients depend only on the degrees and the arities of the maps involved (for instance a Gerstenhaber bracket, a brace $x_1\{x_2,\ldots,x_r\}$, or one of the maps $l_k$ of \cite[Section~5.1]{song}). Then, for all collections $f_1,\ldots,f_r\in\mathrm{End}^\Omega_V$,
\[\iota\big(F(f_1,\ldots,f_r)\big)=F\big(\iota(f_1),\ldots,\iota(f_r)\big),\]
where the left hand side is computed in $\mathrm{End}^\Omega_V$ and the right hand side in $\mathrm{End}_{V_\Omega}$. In particular, $\iota(\mathrm{End}^\Omega_V)$ and $\iota(\mathrm{End}_V)$ are closed under such expressions, and $F(f_1,\ldots,f_r)=0$ if and only if $F(\iota(f_1),\ldots,\iota(f_r))=0$.
\end{lemma}

\begin{proof}
By Lemma~\ref{lem:homog} and its graded version, $\iota(I)=\id_{V_\Omega}$ and $\iota$ commutes with partial compositions; moreover, $\iota$ is linear and preserves degrees and arities, so it preserves the coefficients of $F$. The identity follows by induction on the number of partial compositions in $F$. Partial compositions of constant collections are constant, which gives the closedness of $\iota(\mathrm{End}_V)$, and the last assertion follows from the injectivity of $\iota$.
\end{proof}

\subsection{A new graded Lie algebra}
Given an associative algebra $B$, we construct in this subsection a graded Lie algebra whose Maurer--Cartan elements are the Nijenhuis families on $B$.

Let $(B,\mu)$ be an associative algebra, $\mu(x,y)=x\cdot y$. To lighten the formulas, we omit the indices when they are determined by the inputs: for $f\in\Hom_\Omega(B^{\ot n},B)$, $f(x_1,\ldots,x_n)$ stands for $f_{\alpha_1,\ldots,\alpha_n}(x_1,\ldots,x_n)$, where $\alpha_j$ is the index attached to $x_j$, and an output inserted into an input carries the product of the indices of its own inputs. For $f\in\Hom_\Omega(B^{\ot n},B)$ and $g\in\Hom_\Omega(B^{\ot k},B)$, we define $\llbracket f,g\rrbracket\in\Hom_\Omega(B^{\ot n+k},B)$ by
\begin{equation}\label{eq:MC-bracket}
\begin{aligned}
&\llbracket f,g\rrbracket(x_1,\ldots,x_{n+k})\\
&\quad=(-1)^{nk}f(x_1,\ldots,x_n)\cdot g(x_{n+1},\ldots,x_{n+k})\\
&\qquad-g(x_1,\ldots,x_k)\cdot f(x_{k+1},\ldots,x_{n+k})\\
&\qquad+(-1)^{k-1}\sum_{i=1}^n(-1)^{(i-1)k}f\big(x_1,\ldots,x_{i-1},\delta_\Hoch(g)(x_i,\ldots,x_{i+k}),x_{i+k+1},\ldots,x_{n+k}\big)\\
&\qquad-(-1)^{nk+n-1}\sum_{i=1}^k(-1)^{(i-1)n}g\big(x_1,\ldots,x_{i-1},\delta_\Hoch(f)(x_i,\ldots,x_{i+n}),x_{i+n+1},\ldots,x_{n+k}\big),
\end{aligned}
\end{equation}
where $\delta_\Hoch$ is the Hochschild coboundary of $B$ with coefficients in the regular bimodule, extended to collections as in Section~\ref{sec:background}. For instance, the index of $f$ in the third term is $(\alpha_1,\ldots,\alpha_{i-1},\alpha_i\cdots\alpha_{i+k},\alpha_{i+k+1},\ldots,\alpha_{n+k})$. 

The bracket \eqref{eq:MC-bracket} gives a graded Lie algebra structure.

\begin{lemma}\label{lem:Lie-B}
Let $(B,\mu)$ be an associative algebra. Then $\big(\bigoplus_{n\geq1}\Hom_\Omega(B^{\ot n},B),\llbracket\ ,\ \rrbracket\big)$, with $\Hom_\Omega(B^{\ot n},B)$ in degree $n$, is a graded Lie algebra, that is, for $f$, $g$ and $h$ of arities $n$, $k$ and $l$,
\[\llbracket f,g\rrbracket=-(-1)^{nk}\llbracket g,f\rrbracket,\qquad \llbracket f,\llbracket g,h\rrbracket\rrbracket=\llbracket\llbracket f,g\rrbracket,h\rrbracket+(-1)^{nk}\llbracket g,\llbracket f,h\rrbracket\rrbracket.\]
\end{lemma}

\begin{proof}
For $\Omega$ trivial, this is \cite[Proposition~5.6(i)]{song}: for an associative algebra $C$, \eqref{eq:MC-bracket} is the twisted bracket $l_2$ of \cite[Proof of Proposition~5.6]{song}. We apply this to $B_\Omega=B\ot\bk\Omega$, with the associative product $\iota(\mu)$. The bracket \eqref{eq:MC-bracket} is built from $f$, $g$ and $\mu$ by partial compositions, since $u\cdot v=(\mu\circ_1u)\circ_{p+1}v$ for $u$ of arity $p$, and $\delta_\Hoch(g)=(-1)^{k-1}[\mu,g]_\Omega$ for $g$ of arity $k$ by \eqref{eq:delta-bracket}. By Lemma~\ref{lem:transport}, $\iota\big(\llbracket f,g\rrbracket\big)=\llbracket\iota(f),\iota(g)\rrbracket_{B_\Omega}$. Since $\iota$ is injective and preserves degrees, the graded Lie algebra axioms for $\llbracket\ ,\ \rrbracket$ follow from those for $B_\Omega$.
\end{proof}

The following result characterizes Nijenhuis families as Maurer--Cartan elements.

\begin{thm}\label{thm:MC-family}
Let $(B,\mu)$ be an associative algebra. A collection $N=\{N_\omega\}_{\omega\in\Omega}\in\Hom_\Omega(B,B)$ is a Nijenhuis family on $B$ if and only if $N$ is a Maurer--Cartan element of the graded Lie algebra $\big(\bigoplus_{n\geq1}\Hom_\Omega(B^{\ot n},B),\llbracket\ ,\ \rrbracket\big)$ of Lemma~\ref{lem:Lie-B}, that is, $\llbracket N,N\rrbracket=0$.
\end{thm}

\begin{proof}
Take $f=g=N$ in \eqref{eq:MC-bracket}, so $n=k=1$. The first two terms give $-N_\alpha(x)\cdot N_\beta(y)-N_\alpha(x)\cdot N_\beta(y)$, and the last two terms give
\[N_{\alpha\beta}\big(\delta_\Hoch(N)_{\alpha,\,\beta}(x,y)\big)+N_{\alpha\beta}\big(\delta_\Hoch(N)_{\alpha,\,\beta}(x,y)\big),\]
where the index $\alpha\beta$ of $N$ is the product of the indices of the inputs of $\delta_\Hoch(N)$. Since $\delta_\Hoch(N)_{\alpha,\,\beta}(x,y)=x\cdot N_\beta(y)-N_{\alpha\beta}(x\cdot y)+N_\alpha(x)\cdot y$, we get
\begin{align*}
\llbracket N,N\rrbracket_{\alpha,\,\beta}(x,y)&=-2\,N_\alpha(x)\cdot N_\beta(y)+2\,N_{\alpha\beta}\big(\delta_\Hoch(N)_{\alpha,\,\beta}(x,y)\big)\\
&=-2\Big(N_\alpha(x)\cdot N_\beta(y)-N_{\alpha\beta}\big(N_\alpha(x)\cdot y+x\cdot N_\beta(y)-N_{\alpha\beta}(x\cdot y)\big)\Big),
\end{align*}
for all $x,y\in B$ and $\alpha,\beta\in\Omega$. Hence $\llbracket N,N\rrbracket=0$ if and only if $N$ satisfies \eqref{eq:NF-intro}.
\end{proof}

Let $(A,\{N_\omega\}_{\omega\in\Omega})$ be a Nijenhuis family algebra. By Theorem~\ref{thm:MC-family}, $f\mapsto(-1)^n\llbracket N,f\rrbracket$ is a differential on $\bigoplus_{n\geq1}\Hom_\Omega(A^{\ot n},A)$. Let $(M,\{N_{M,\,\omega}\}_{\omega\in\Omega})$ be a Nijenhuis family bimodule over $(A,\{N_\omega\})$. We now describe the corresponding differential explicitly. For $a_1,\ldots,a_{n+1}\in A$, $\alpha_1,\ldots,\alpha_{n+1}\in\Omega$, $1\leq i\leq n$ and $y\in A$, we write
\[\mathbf a^{(i)}(y):=(a_1,\ldots,a_{i-1},y,a_{i+2},\ldots,a_{n+1}),\qquad \boldsymbol\alpha^{(i)}:=(\alpha_1,\ldots,\alpha_{i-1},\alpha_i\alpha_{i+1},\alpha_{i+2},\ldots,\alpha_{n+1}).\]
For $n\geq0$, put $C^n_\NF(A,M):=\Hom_\Omega(A^{\ot n},M)$, where $\Hom_\Omega(A^{\ot 0},M):=M$, and define $\delta^n_{\NF,M}:C^n_\NF(A,M)\to C^{n+1}_\NF(A,M)$ by the following formula, where $\omega:=\alpha_1\cdots\alpha_{n+1}$:
\begin{align*}
&\big(\delta^n_{\NF,M}(f)\big)_{\alpha_1,\ldots,\alpha_{n+1}}(a_1,\ldots,a_{n+1})\\
={}&N_{\alpha_1}(a_1)\cdot f_{\alpha_2,\ldots,\alpha_{n+1}}(a_2,\ldots,a_{n+1})-N_{M,\,\omega}\big(a_1\cdot f_{\alpha_2,\ldots,\alpha_{n+1}}(a_2,\ldots,a_{n+1})\big)\\
&+\sum_{i=1}^n(-1)^{i}\Big(f_{\boldsymbol\alpha^{(i)}}\big(\mathbf a^{(i)}(a_i\cdot_{N;\alpha_i,\,\alpha_{i+1}}a_{i+1})\big)-N_{M,\,\omega}\,f_{\boldsymbol\alpha^{(i)}}\big(\mathbf a^{(i)}(a_i\cdot a_{i+1})\big)\Big)\\
&+(-1)^{n+1}\Big(f_{\alpha_1,\ldots,\alpha_n}(a_1,\ldots,a_n)\cdot N_{\alpha_{n+1}}(a_{n+1})-N_{M,\,\omega}\big(f_{\alpha_1,\ldots,\alpha_n}(a_1,\ldots,a_n)\cdot a_{n+1}\big)\Big).
\end{align*}
For $n=0$ and $m\in M$, this formula reads
\begin{equation}\label{eq:delta0-NF}
\delta^0_{\NF,M}(m)_\alpha(a)=N_\alpha(a)\cdot m-m\cdot N_\alpha(a)-N_{M,\,\alpha}(a\cdot m-m\cdot a),\qquad a\in A,\ \alpha\in\Omega.
\end{equation}
Note that $0$-cochains carry no index, while $\delta^0_{\NF,M}(m)$ is a collection with one index.
In other words,
\begin{equation}\label{eq:partial-short}
\delta^n_{\NF,M}(f)=-N_M\circ\delta^n_{\Hoch,M}(f)+\partial^n(f),
\end{equation}
where $\partial^n=\partial$ is the Hochschild differential of the associative algebra $A_N$ with coefficients in $M_N$ (Proposition~\ref{prop:MP}), and
\[(N_M\circ g)_{\alpha_1,\ldots,\alpha_{n+1}}:=N_{M,\,\alpha_1\cdots\alpha_{n+1}}\circ g_{\alpha_1,\ldots,\alpha_{n+1}}\]
for $g\in\Hom_\Omega(A^{\ot n+1},M)$.

\begin{prop}\label{prop:partial-bracket}
Let $(B,\{N_\omega\})$ be a Nijenhuis family algebra, and let $\delta_\NF$ be the above differential with $A=M=B$. Then $\delta^n_\NF(f)=(-1)^n\llbracket N,f\rrbracket$ for all $n\geq1$ and $f\in\Hom_\Omega(B^{\ot n},B)$.
\end{prop}

\begin{proof}
Let $f\in\Hom_\Omega(B^{\ot n},B)$. We apply \eqref{eq:MC-bracket} with $N$ and $f$ in place of $f$ and $g$, that is, with arities $1$ and $n$ instead of $n$ and $k$. Recall that
\[\delta_\Hoch(N)_{\alpha,\,\beta}(x,y)=x\cdot N_\beta(y)-N_{\alpha\beta}(x\cdot y)+N_\alpha(x)\cdot y=x\cdot_{N;\alpha,\,\beta}y.\]
Put $\omega=\alpha_1\cdots\alpha_{n+1}$. Computing the four terms of \eqref{eq:MC-bracket} and multiplying by $(-1)^n$, we obtain
\begin{align*}
(-1)^n\llbracket N,f\rrbracket_{\alpha_1,\ldots,\alpha_{n+1}}(x_1,\ldots,x_{n+1})
={}&N_{\alpha_1}(x_1)\cdot f_{\alpha_2,\ldots,\alpha_{n+1}}(x_2,\ldots,x_{n+1})\\
&+\sum_{i=1}^n(-1)^{i}f_{\boldsymbol\alpha^{(i)}}\big(\mathbf x^{(i)}(x_i\cdot_{N;\alpha_i,\,\alpha_{i+1}}x_{i+1})\big)\\
&+(-1)^{n+1}f_{\alpha_1,\ldots,\alpha_n}(x_1,\ldots,x_n)\cdot N_{\alpha_{n+1}}(x_{n+1})\\
&-N_{\omega}\big(\delta_\Hoch(f)_{\alpha_1,\ldots,\alpha_{n+1}}(x_1,\ldots,x_{n+1})\big)\\
={}&\big(\partial(f)-N\circ\delta^n_\Hoch(f)\big)_{\alpha_1,\ldots,\alpha_{n+1}}(x_1,\ldots,x_{n+1}),
\end{align*}
which is $\delta_\NF^n(f)$ by \eqref{eq:partial-short}.
\end{proof}

In particular, by Theorem~\ref{thm:MC-family} and the graded Jacobi identity, we have $\delta^{n+1}_\NF\circ\delta^n_\NF=0$ for the regular bimodule. For an arbitrary Nijenhuis family bimodule $M$, regard $f\in\Hom_\Omega(A^{\ot n},M)$ as the collection $\tilde f\big((a_1,m_1),\ldots,(a_n,m_n)\big):=\big(0,f(a_1,\ldots,a_n)\big)$ on $A\ltimes M$, as in \cite[Section~2.4]{song}. Since $M\cdot M=0$, the differential $\delta_\NF$ of $(A\ltimes M,\widehat N)$ satisfies $\delta_\NF(\tilde f)=\widetilde{\delta_\NF(f)}$. Hence $\delta^{n+1}_\NF\circ\delta^n_\NF=0$ for all $n\geq1$; the case $n=0$ is proved in Proposition~\ref{prop:Phi-chain} below.

The cochain complex $(C^\bullet_\NF(A,M),\delta^\bullet_\NF)$ is called the \emph{cochain complex of the Nijenhuis family operator} $\{N_\omega\}_{\omega\in\Omega}$ with coefficients in $(M,\{N_{M,\,\omega}\}_{\omega\in\Omega})$, and its cohomology, denoted by $\rmH^\bullet_{\NF}(A,M)$, the \emph{cohomology of the Nijenhuis family operator} $\{N_\omega\}$ \emph{with coefficients in} $(M,\{N_{M,\,\omega}\})$. When $M=A$, we write $\rmH^\bullet_{\NF}(A)$.

\subsection{A new \texorpdfstring{$L_\infty$}{L-infinity}-algebra}\label{subsec:Linfty}
In this subsection, we describe an $L_\infty$-algebra whose Maurer--Cartan elements are Nijenhuis family algebras, in which both the product and the family of operators are variable. This $L_\infty$-algebra is the family version of the $L_\infty$-algebra of \cite[Section~5]{song}.

Let $V$ be a vector space, and consider the graded spaces
\[\mathfrak h_V:=\bigoplus_{n\geq0}\Hom(V^{\ot n+1},V),\qquad \mathfrak a_V:=\bigoplus_{n\geq0}\Hom_\Omega(V^{\ot n+1},V),\]
with $\Hom(V^{\ot n+1},V)$ and $\Hom_\Omega(V^{\ot n+1},V)$ in degree $n$, so that
\[(s^{-1}\mathfrak h_V\oplus\mathfrak a_V)_{-1}=\Hom(V,V),\qquad (s^{-1}\mathfrak h_V\oplus\mathfrak a_V)_n=\Hom(V^{\ot n+2},V)\oplus\Hom_\Omega(V^{\ot n+1},V),\quad n\geq0.\]

The following result describes the $L_\infty$-algebra structure on $s^{-1}\mathfrak h_V\oplus\mathfrak a_V$.

\begin{thm}\label{thm:Linfty}
The graded space $s^{-1}\mathfrak h_V\oplus\mathfrak a_V$ carries an $L_\infty$-algebra structure $\{l_k\}_{k\geq1}$ such that:
\begin{enumerate}[(i)]
\item $l_1=0$ and $l_2(s^{-1}q,s^{-1}q')=(-1)^{|q|}s^{-1}[q,q']_G$ for homogeneous $q,q'\in\mathfrak h_V$;
\item for $q\in\Hom(V^{\ot k},V)$ and $f_1,\ldots,f_k\in\mathfrak a_V$, the element $l_{k+1}(s^{-1}q,f_1,\ldots,f_k)\in\mathfrak a_V$ is given by the brace formula of \cite[Section~5.1]{song}, computed in the operad $\{\Hom_\Omega(V^{\ot n},V)\}_{n\geq1}$; in particular, if $q=\mu$ is associative, then $l_3(s^{-1}\mu,f,g)$ corresponds to the bracket $\llbracket f,g\rrbracket$ of \eqref{eq:MC-bracket};
\item up to permutations of the inputs, all the other maps $l_k$ vanish. In particular, if $q$ has arity $k$ and $j\neq k$, then $l_{j+1}(s^{-1}q,f_1,\ldots,f_j)=0$.
\end{enumerate}
\end{thm}

\begin{proof}
For $\Omega$ trivial, this is the $L_\infty$-algebra $\mathcal C_{\mathrm{NjA}}(W)$ of \cite[Definition~5.1]{song}, described in \cite[Section~5.1]{song}: $l_2$ on $s^{-1}\mathfrak h_W$ is the Gerstenhaber bracket, and the maps $l_{k+1}(s^{-1}q,f_1,\ldots,f_k)$ are brace formulas, that is, sums of iterated partial compositions, whose coefficients (including the Koszul signs of the graded symmetrization) only depend on the degrees and the arities of the arguments. Take $W=V_\Omega$ and consider the injective map $(s^{-1}q,f)\mapsto\big(s^{-1}\iota(q),\iota(f)\big)$ from $s^{-1}\mathfrak h_V\oplus\mathfrak a_V$ to $s^{-1}\mathfrak h_{V_\Omega}\oplus\mathfrak a_{V_\Omega}$, where $q$ is regarded as a constant collection. By Lemma~\ref{lem:transport}, the image is closed under the maps $l_k$, with $l_{k+1}(s^{-1}\iota(q),\iota(f_1),\ldots,\iota(f_k))=\iota\big(l_{k+1}(s^{-1}q,f_1,\ldots,f_k)\big)$. Thus the image is an $L_\infty$-subalgebra, and transporting its structure gives (i)--(iii). The last assertion of (ii) follows from \cite[Proof of Proposition~5.6]{song}.
\end{proof}

The interest of this $L_\infty$-algebra comes from the following Maurer--Cartan characterization of Nijenhuis family algebras.

\begin{thm}\label{thm:MC-NFalg}
Let $A$ be a vector space, $\mu:A\ot A\to A$ a linear map and $N=\{N_\omega:A\to A\}_{\omega\in\Omega}$ a collection of linear maps. Then $\big((A,\mu),\{N_\omega\}\big)$ is a Nijenhuis family algebra if and only if $\theta=(s^{-1}\mu,N)\in(s^{-1}\mathfrak h_A\oplus\mathfrak a_A)_0$ is a Maurer--Cartan element of the $L_\infty$-algebra of Theorem~\ref{thm:Linfty} for $V=A$.
\end{thm}

\begin{proof}
By Theorem~\ref{thm:Linfty}, $l_1(\theta)=0$. Since $\mu$ has arity $2$, the only nonzero values of the maps $l_k$ on $\theta$ come from $l_2(s^{-1}\mu,s^{-1}\mu)$ and from $l_3(s^{-1}\mu,N,N)$: indeed, $l_2(s^{-1}\mu,N)=0$ and $l_2(N,N)=0$ by (iii), and $l_k(s^{-1}\mu,N,\ldots,N)$ with $k-1\neq2$ copies of $N$ vanishes. By graded symmetry, $l_3(\theta,\theta,\theta)=3\,l_3(s^{-1}\mu,N,N)$. Hence
\[\sum_{k\geq1}\frac1{k!}l_k(\theta,\ldots,\theta)=\frac12\,l_2(s^{-1}\mu,s^{-1}\mu)+\frac12\,l_3(s^{-1}\mu,N,N)=\Big(-\frac12\,s^{-1}[\mu,\mu]_G,\ \frac12\,l_3(s^{-1}\mu,N,N)\Big).\]
The two components lie in different summands, so $\theta$ is a Maurer--Cartan element if and only if $[\mu,\mu]_G=0$ and $l_3(s^{-1}\mu,N,N)=0$. The first condition means that $\mu$ is associative. Under this condition, $l_3(s^{-1}\mu,N,N)$ corresponds to $\llbracket N,N\rrbracket$ by Theorem~\ref{thm:Linfty}(ii), and $\llbracket N,N\rrbracket=0$ if and only if $\{N_\omega\}$ is a Nijenhuis family, by Theorem~\ref{thm:MC-family}.
\end{proof}

\section{Cohomology of Nijenhuis family algebras}\label{sec:cohomology}

In this section, $(A,\{N_\omega\}_{\omega\in\Omega})$ is a Nijenhuis family algebra and $(M,\{N_{M,\,\omega}\}_{\omega\in\Omega})$ is a Nijenhuis family bimodule over it. We combine the Hochschild cohomology of $A$ with coefficients in $M$ and the cohomology of the Nijenhuis family $\{N_\omega\}$ with coefficients in $M$ to define the cohomology of the Nijenhuis family algebra with coefficients in $M$.

\subsection{Hochschild cohomology}
For every $n\geq0$, we put $C^n_\Hoch(A,M):=\Hom(A^{\ot n},M)$, with $\Hom(A^{\ot0},M)=M$, and let $\delta^n_{\Hoch,M}$ be the Hochschild coboundary of Section~\ref{sec:background}. In low degrees,
\begin{align*}
\delta^0_{\Hoch,M}(m)(a)&=a\cdot m-m\cdot a,\\
\delta^1_{\Hoch,M}(f)(a,b)&=a\cdot f(b)-f(a\cdot b)+f(a)\cdot b,\\
\delta^2_{\Hoch,M}(f)(a,b,c)&=a\cdot f(b,c)-f(a\cdot b,c)+f(a,b\cdot c)-f(a,b)\cdot c.
\end{align*}
The cohomology of $\big(C^\bullet_\Hoch(A,M),\delta_{\Hoch,M}\big)$ is the Hochschild cohomology $\HH^\bullet(A,M)$.

\subsection{Cohomology of Nijenhuis family operators}
We use the cochain complex $C^\bullet_\NF(A,M)=\Hom_\Omega(A^{\ot\bullet},M)$ of the Nijenhuis family operator $\{N_\omega\}$ with coefficients in $(M,\{N_{M,\,\omega}\})$, defined in Section~\ref{sec:MC}, with cohomology $\rmH^\bullet_\NF(A,M)$. Recall from \eqref{eq:partial-short} that its differential is $\delta_{\NF,M}=-N_M\circ\delta_{\Hoch,M}+\partial$, where $\partial$ is the Hochschild differential of $A_N$ with coefficients in $M_N$.

\subsection{A chain map from the Hochschild complex}
For $n\geq0$, $f\in\Hom_\Omega(A^{\ot n},M)$, $a_1,\ldots,a_n\in A$, $\alpha_1,\ldots,\alpha_n\in\Omega$ and $S\subseteq\{1,\ldots,n\}$, let $\mathbf a^S$ be the tuple obtained from $(a_1,\ldots,a_n)$ by replacing $a_i$ by $N_{\alpha_i}(a_i)$ for every $i\in S$, and define $\Phi^n_M(f)\in\Hom_\Omega(A^{\ot n},M)$ by
\begin{equation}\label{eq:Phi}
\Phi^n_M(f)_{\alpha_1, \dots, \alpha_n}(a_1,\dots,  a_n):=\sum_{S\subseteq\{1,\ldots,n\}}\big(-N_{M,\,\alpha_1\cdots\alpha_n}\big)^{n-|S|}\Big(f_{\alpha_1,\ldots,\alpha_n}(\mathbf a^S)\Big).
\end{equation}
For instance, $\Phi^1_M(f)_\alpha(a)=f_\alpha(N_\alpha(a))-N_{M,\,\alpha}(f_\alpha(a))$ and
\[\Phi^2_M(f)_{\alpha,\,\beta}(a,b)=f_{\alpha,\,\beta}(N_\alpha(a),N_\beta(b))-N_{M,\,\alpha\beta}\big(f_{\alpha,\,\beta}(N_\alpha(a),b)+f_{\alpha,\,\beta}(a,N_\beta(b))\big)+N^2_{M,\,\alpha\beta}\big(f_{\alpha,\,\beta}(a,b)\big).\]
If $f\in\Hom(A^{\ot n},M)$ does not depend on the indices, $\Phi^n_M(f)$ does, through the maps $N_{\alpha_i}$ and $N_{M,\,\alpha_1\cdots\alpha_n}$. For $n=0$, $\Phi^0_M=\id_M$. For $\Omega$ trivial, $\Phi$ is the chain map of \cite[Proposition~2.4]{song}.

To prove that $\Phi_M$ is a chain map, we introduce, for $a\in A$ and $\alpha,\beta\in\Omega$, the maps $D^{\alpha,\,\beta}_a,E^{\alpha,\,\beta}_a:M\to M$ given by $D^{\alpha,\,\beta}_a(m):=N_\alpha(a)\cdot m-N_{M,\,\alpha\beta}(a\cdot m)$ and $E^{\alpha,\,\beta}_a(m):=m\cdot N_\beta(a)-N_{M,\,\alpha\beta}(m\cdot a)$.

\begin{lemma}\label{lem:commute}
For all $a\in A$ and $\alpha,\beta\in\Omega$, we have $D^{\alpha,\,\beta}_a\circ N_{M,\,\beta}=N_{M,\,\alpha\beta}\circ D^{\alpha,\,\beta}_a$ and $E^{\alpha,\,\beta}_a\circ N_{M,\,\alpha}=N_{M,\,\alpha\beta}\circ E^{\alpha,\,\beta}_a$.
\end{lemma}

\begin{proof}
By \eqref{eq:left},
\begin{align*}D^{\alpha,\,\beta}_a\big(N_{M,\,\beta}(m)\big)&=N_\alpha(a)\cdot N_{M,\,\beta}(m)-N_{M,\,\alpha\beta}\big(a\cdot N_{M,\,\beta}(m)\big)\\
&=N_{M,\,\alpha\beta}\big(N_\alpha(a)\cdot m-N_{M,\,\alpha\beta}(a\cdot m)\big).\end{align*}
The second identity follows in the same way from \eqref{eq:right}.
\end{proof}

By induction, $D^{\alpha,\,\beta}_a\circ(-N_{M,\,\beta})^k=(-N_{M,\,\alpha\beta})^k\circ D^{\alpha,\,\beta}_a$ and $E^{\alpha,\,\beta}_a\circ(-N_{M,\,\alpha})^k=(-N_{M,\,\alpha\beta})^k\circ E^{\alpha,\,\beta}_a$ for all $k\geq0$. For $f\in C^n_\NF(A,M)$, $\beta=\alpha_2\cdots\alpha_{n+1}$, $\gamma=\alpha_1\cdots\alpha_n$ and $\omega=\alpha_1\cdots\alpha_{n+1}$, the first and last terms of $\delta_{\NF,M}(f)(a_1,\ldots,a_{n+1})$ are $D^{\alpha_1,\,\beta}_{a_1}\big(f(a_2,\ldots,a_{n+1})\big)$ and $(-1)^{n+1}E^{\gamma,\,\alpha_{n+1}}_{a_{n+1}}\big(f(a_1,\ldots,a_n)\big)$.

\begin{prop}\label{prop:Phi-chain}
For all $n\geq0$ and $f\in\Hom_\Omega(A^{\ot n},M)$, we have $\delta^n_{\NF,M}\big(\Phi^n_M(f)\big)=\Phi^{n+1}_M\big(\delta^n_{\Hoch,M}(f)\big)$. In particular, $\delta^1_{\NF,M}\circ\delta^0_{\NF,M}=0$, and $\Phi_M$ is a chain map from $\big(C^\bullet_\Hoch(A,M),\delta_{\Hoch,M}\big)$ to $\big(C^\bullet_\NF(A,M),\delta^\bullet_{\NF,M}\big)$.
\end{prop}

\begin{proof}
For $n=0$ and $m\in M$, we have
\begin{align*}\Phi^1_M\big(\delta^0_{\Hoch,M}(m)\big)_\alpha(a)&=\delta^0_{\Hoch,M}(m)\big(N_\alpha(a)\big)-N_{M,\,\alpha}\big(\delta^0_{\Hoch,M}(m)(a)\big)\\
&=N_\alpha(a)\cdot m-m\cdot N_\alpha(a)-N_{M,\,\alpha}(a\cdot m-m\cdot a),\end{align*}
which is $\delta^0_{\NF,M}\big(\Phi^0_M(m)\big)_\alpha(a)$ by \eqref{eq:delta0-NF}. Let now $n\geq1$.
By \eqref{eq:Phi},
\begin{align*}\Phi^{n+1}_M\big(\delta^n_{\Hoch,M}(f)\big)(a_1,\ldots,a_{n+1})&=\sum_{S}(-N_{M,\,\omega})^{n+1-|S|}\big(\delta^n_{\Hoch,M}(f)(\mathbf a^S)\big)\\
&=T_1+\sum_{i=1}^nT_2^{(i)}+T_3,\end{align*}
where $T_1$, $T_2^{(i)}$ and $T_3$ collect the contributions of the first, the $i$-th middle and the last term of $\delta^n_{\Hoch,M}(f)$. We show that they coincide with the corresponding terms of $\delta^n_{\NF,M}\big(\Phi^n_M(f)\big)$. For $S'\subseteq\{2,\ldots,n+1\}$, put $F_{S'}:=f\big((a_2,\ldots,a_{n+1})^{S'}\big)$ and $k_{S'}:=n-|S'|$. Splitting the sum according to whether $1\in S$, and using Lemma~\ref{lem:commute}, we get
\[T_1=\sum_{S'}(-N_{M,\,\omega})^{k_{S'}}\big(N_{\alpha_1}(a_1)\cdot F_{S'}-N_{M,\,\omega}(a_1\cdot F_{S'})\big)=D^{\alpha_1,\,\beta}_{a_1}\big(\Phi^n_M(f)(a_2,\ldots,a_{n+1})\big),\]
and, in the same way, $T_3=(-1)^{n+1}E^{\gamma,\,\alpha_{n+1}}_{a_{n+1}}\big(\Phi^n_M(f)(a_1,\ldots,a_n)\big)$.

Fix $1\le i\le n$ and write $a=a_i$, $b=a_{i+1}$, $N_M=N_{M,\,\omega}$, and $N$ for the operator whose index is determined by its argument. For $T\subseteq\{1,\ldots,n+1\}\setminus\{i,i+1\}$ with $j=|T|$ and $y\in A$, let $F_T(y):=f_{\boldsymbol\alpha^{(i)}}\big(\mathbf a^{(i)}_T(y)\big)$, where $\mathbf a^{(i)}_T(y)$ is obtained from $\mathbf a^{(i)}(y)$ by replacing $a_\ell$ by $N_{\alpha_\ell}(a_\ell)$ for $\ell\in T$. By \eqref{eq:Phi},
\begin{equation}\label{eq:Phi-slot}
\Phi^n_M(f)_{\boldsymbol\alpha^{(i)}}\big(\mathbf a^{(i)}(y)\big)=\sum_T(-N_M)^{n-1-j}\Big(F_T\big(N(y)\big)-N_M F_T(y)\Big),\qquad \forall\, y\in A.
\end{equation}
Writing $S=T\cup U$ with $U\subseteq\{i,i+1\}$, and using $N(a)\cdot N(b)=N(a\cdot_N b)$ (see \eqref{eq:Nstar}) and $N(a)\cdot b+a\cdot N(b)=a\cdot_N b+N(a\cdot b)$, we get
\begin{align*}
T_2^{(i)}&=(-1)^i\sum_T(-N_M)^{n-1-j}\Big(F_T\big(N(a)\cdot N(b)\big)-N_MF_T\big(N(a)\cdot b+a\cdot N(b)\big)+N_M^2F_T(a\cdot b)\Big)\\
&=(-1)^i\Big(\Phi^n_M(f)_{\boldsymbol\alpha^{(i)}}\big(\mathbf a^{(i)}(a\cdot_N b)\big)-N_M\,\Phi^n_M(f)_{\boldsymbol\alpha^{(i)}}\big(\mathbf a^{(i)}(a\cdot b)\big)\Big),
\end{align*}
by \eqref{eq:Phi-slot} with $y=a\cdot_N b$ and $y=a\cdot b$. This proves the first assertion; the last one follows, since $f\in\Hom(A^{\ot n},M)$ is regarded as a constant collection.
\end{proof}

For $m\in M$, $\delta^0_{\Hoch,M}(m)\in\Hom(A,M)$ does not depend on the indices, so, by the case $n=1$,
\[\delta^1_{\NF,M}\big(\delta^0_{\NF,M}(m)\big)=\delta^1_{\NF,M}\big(\Phi^1_M(\delta^0_{\Hoch,M}(m))\big)=\Phi^2_M\big(\delta^1_{\Hoch,M}\delta^0_{\Hoch,M}(m)\big)=0.\]
This proves the assertion $\delta^1_{\NF,M}\circ\delta^0_{\NF,M}=0$ of Proposition~\ref{prop:Phi-chain}.

\subsection{Cohomology of Nijenhuis family algebras with coefficients}
We now combine the Hochschild cochain complex and the cochain complex of the Nijenhuis family operator, by means of the chain map $\Phi_M$.

\begin{defn}\label{def:cone}
The \emph{cochain complex of the Nijenhuis family algebra} $(A,\{N_\omega\})$ with coefficients in $(M,\{N_{M,\,\omega}\})$ is the mapping cone of $\Phi$, shifted by $-1$:
\begin{gather*}C^0_\NFa(A,M):=C^0_\Hoch(A,M)=M,\\ C^n_\NFa(A,M):=C^n_\Hoch(A,M)\oplus C^{n-1}_\NF(A,M)=\Hom(A^{\ot n},M)\oplus\Hom_\Omega(A^{\ot n-1},M),\quad n\geq1,\end{gather*}
with the differential $\delta^n_{\NFa,M}:C^n_\NFa(A,M)\to C^{n+1}_\NFa(A,M)$ given by
\[\delta^0_{\NFa,M}(m):=\big(\delta^0_{\Hoch,M}(m),-m\big),\qquad \delta^n_{\NFa,M}(f,g):=\big(\delta^n_{\Hoch,M}(f),\,-\delta^{n-1}_{\NF,M}(g)-\Phi^n_M(f)\big),\quad n\geq1.\]
Its cohomology, denoted by $\rmH^\bullet_\NFa(A,M)$, is called the \emph{cohomology of the Nijenhuis family algebra} $(A,\{N_\omega\})$ with coefficients in $(M,\{N_{M,\,\omega}\})$. When $M=A$ is the regular bimodule, we write $C^\bullet_\NFa(A)$ and $\rmH^\bullet_\NFa(A)$.
\end{defn}

The following result shows that Definition~\ref{def:cone} indeed gives a cochain complex.

\begin{prop}\label{prop:cone}
$\delta^{n+1}_{\NFa,M}\circ \delta^n_{\NFa,M}=0$ for all $n\geq0$.
\end{prop}

\begin{proof}
For $m\in M$, since $\Phi^0_M=\id_M$, Proposition~\ref{prop:Phi-chain} gives
\[\delta^1_{\NFa,M}\big(\delta^0_{\NFa,M}(m)\big)=\Big(\delta^1_{\Hoch,M}\delta^0_{\Hoch,M}(m),\ \delta^0_{\NF,M}(m)-\Phi^1_M\big(\delta^0_{\Hoch,M}(m)\big)\Big)=(0,0).\]
For $n\geq1$ and $(f,g)\in C^n_\NFa(A,M)$,
\begin{align*}
\delta^{n+1}_{\NFa,M}\big(\delta^n_{\NFa,M}(f,g)\big)
&=\delta^{n+1}_{\NFa,M}\big(\delta^n_{\Hoch,M}(f),\,-\delta^{n-1}_{\NF,M}(g)-\Phi^n_M(f)\big)\\
&=\Big(\delta^{n+1}_{\Hoch,M}\delta^n_{\Hoch,M}(f),\ \delta^n_{\NF,M}\delta^{n-1}_{\NF,M}(g)+\delta^n_{\NF,M}\Phi^n_M(f)-\Phi^{n+1}_M\delta^n_{\Hoch,M}(f)\Big),
\end{align*}
which vanishes since $\delta_{\Hoch,M}^2=0$, $\delta_{\NF,M}^2=0$ and $\delta_{\NF,M}\circ\Phi_M=\Phi_M\circ\delta_{\Hoch,M}$.
\end{proof}

In low degrees, an element $(\psi,\chi)\in C^2_\NFa(A,M)=\Hom(A^{\ot 2},M)\oplus\Hom_\Omega(A,M)$ is a $2$-cocycle if and only if $\delta^2_{\Hoch,M}(\psi)=0$ and $\delta^1_{\NF,M}(\chi)+\Phi^2_M(\psi)=0$, that is, for all $a,b\in A$ and $\alpha,\beta\in\Omega$,
\begin{multline}\label{eq:2-cocycle}
N_\alpha(a)\cdot\chi_\beta(b)+\chi_\alpha(a)\cdot N_\beta(b)-N_{M,\,\alpha\beta}\big(a\cdot\chi_\beta(b)+\chi_\alpha(a)\cdot b\big)-\chi_{\alpha\beta}(a\cdot_{N;\alpha,\beta}b)+N_{M,\,\alpha\beta}\,\chi_{\alpha\beta}(a\cdot b)\\
+\psi\big(N_\alpha(a),N_\beta(b)\big)-N_{M,\,\alpha\beta}\Big(\psi\big(N_\alpha(a),b\big)+\psi\big(a,N_\beta(b)\big)\Big)+N_{M,\,\alpha\beta}^2\,\psi(a,b)=0.
\end{multline}

The short sequence of cochain complexes
\[0\to\big(C^{\bullet-1}_\NF(A,M),-\delta_{\NF,M}\big)\xrightarrow{\ i\ }\big(C^\bullet_\NFa(A,M),\delta_{\NFa,M}\big)\xrightarrow{\ p\ }\big(C^\bullet_\Hoch(A,M),\delta_{\Hoch,M}\big)\to0,\]
where $i(g)=(0,g)$ and $p(f,g)=f$,
is exact, and $i$, $p$ are chain maps: $\delta_{\NFa,M}(0,g)=(0,-\delta_{\NF,M}(g))=i(-\delta_{\NF,M}(g))$ and $p\circ\delta_{\NFa,M}=\delta_{\Hoch,M}\circ p$. Since the complexes $(C^\bullet_\NF,\delta_{\NF,M})$ and $(C^\bullet_\NF,-\delta_{\NF,M})$ have the same cohomology, we obtain the following result.

\begin{thm}\label{thm:les}
There is a long exact sequence
\begin{multline*}0\to\rmH^0_\NFa(A,M)\to\HH^0(A,M)\to\rmH^0_\NF(A,M)\to\cdots\\\to\rmH^{n-1}_{\NF}(A,M)\xrightarrow{\ i\ }\rmH^n_\NFa(A,M)\xrightarrow{\ p\ }\HH^n(A,M)\xrightarrow{\ -\Phi\ }\rmH^n_{\NF}(A,M)\to\rmH^{n+1}_\NFa(A,M)\to\cdots.\end{multline*}
\end{thm}

\begin{proof}
It remains to compute the connecting map. Let $f\in C^n_\Hoch(A,M)$ be a cocycle. Then $(f,0)$ (or $f$ if $n=0$) is a lift of $f$, and $\delta^n_{\NFa,M}(f,0)=\big(0,-\Phi^n_M(f)\big)=i\big(-\Phi^n_M(f)\big)$ (resp. $\delta^0_{\NFa,M}(f)=(0,-f)$). Hence the connecting map sends $[f]$ to $[-\Phi^n_M(f)]$.
\end{proof}

\begin{remark}\label{rem:H01}
The cohomology in low degrees is as follows. Since $\delta^0_{\NFa,M}(m)=\big(\delta^0_{\Hoch,M}(m),-m\big)$ is injective, $\rmH^0_\NFa(A,M)=0$. A $1$-cochain $(f,x)\in\Hom(A,M)\oplus M$ satisfies $(f,x)+\delta^0_{\NFa,M}(x)=\big(f+\delta^0_{\Hoch,M}(x),0\big)$, so every class in $\rmH^1_\NFa(A,M)$ has a representative of the form $(D,0)$, and this representative is unique, since $\big(\delta^0_{\Hoch,M}(m),-m\big)$ has second component $0$ only for $m=0$. Moreover, $\delta^1_{\NFa,M}(D,0)=\big(\delta^1_{\Hoch,M}(D),-\Phi^1_M(D)\big)$. Hence
\[\rmH^1_\NFa(A,M)\cong\big\{D:A\to M\ \text{derivation}\ \big|\ D\circ N_\omega=N_{M,\,\omega}\circ D\ \text{for all }\omega\in\Omega\big\},\]
the space of derivations compatible with the Nijenhuis families.
\end{remark}

\subsection{Relation with Nijenhuis associative algebras}
By Lemma~\ref{lem:Ntilde}, $\widetilde N(a\ot\omega)=N_\omega(a)\ot\omega$ is a Nijenhuis operator on the associative algebra $A_\Omega:=A\ot\bk\Omega$, with product $(a\ot\alpha)(b\ot\beta)=(a\cdot b)\ot\alpha\beta$. Put $M_\Omega:=M\ot\bk\Omega$, with the actions $(a\ot\alpha)\cdot(m\ot\beta):=(a\cdot m)\ot\alpha\beta$, $(m\ot\beta)\cdot(a\ot\alpha):=(m\cdot a)\ot\beta\alpha$, and $\widetilde N_M(m\ot\omega):=N_{M,\,\omega}(m)\ot\omega$. As in Section~\ref{sec:MC}, for $f\in\Hom_\Omega(A^{\ot n},M)$, $n\geq1$, put
\[\iota(f)(a_1\ot\alpha_1,\ldots,a_n\ot\alpha_n):=f_{\alpha_1,\ldots,\alpha_n}(a_1,\ldots,a_n)\ot\alpha_1\cdots\alpha_n.\]
We compare our complexes with the complexes of the Nijenhuis associative algebra $(A_\Omega,\widetilde N)$ with coefficients in $M_\Omega$ defined in \cite[Section~2]{song}.

\begin{prop}\label{prop:song-embedding}
\begin{enumerate}
\item $(M_\Omega,\widetilde N_M)$ is a Nijenhuis bimodule over the Nijenhuis associative algebra $(A_\Omega,\widetilde N)$.
\item For $n\geq1$, the map $\iota:\Hom_\Omega(A^{\ot n},M)\to\Hom(A_\Omega^{\ot n},M_\Omega)$ is injective, and
\[\iota\circ\delta^n_{\Hoch,M}=\delta^n_{\Hoch,M_\Omega}\circ\iota,\qquad \iota\circ\delta^n_{\NF,M}=\delta^n_{\NF,M_\Omega}\circ\iota,\qquad \iota\circ\Phi^n_M=\Phi^n_{M_\Omega}\circ\iota,\]
where the right hand sides are the differentials and the chain map of \cite[Section~2]{song} for $(A_\Omega,\widetilde N)$ and $M_\Omega$. If $\Omega$ is a monoid with unit $1$, the same holds for $n=0$, with $\iota(m):=m\ot1$.
\item Consequently, for $n\geq2$, $(f,g)\mapsto(\iota(f),\iota(g))$ identifies $C^n_\NFa(A,M)$ with the subspace of $C^n_\NFa(A_\Omega,M_\Omega)$ consisting of the pairs $(\iota(f),\iota(g))$ with $f\in\Hom(A^{\ot n},M)$ a constant collection and $g\in\Hom_\Omega(A^{\ot n-1},M)$, and this identification is compatible with the differentials.
\end{enumerate}
\end{prop}

\begin{proof}
(1) The bimodule axioms of $M_\Omega$ follow from those of $M$, tensored with the products in $\Omega$. For $a\in A$, $m\in M$ and $\alpha,\beta\in\Omega$, both sides of the left Nijenhuis bimodule identity for $(\widetilde N,\widetilde N_M)$, evaluated on $a\ot\alpha$ and $m\ot\beta$, are the two sides of \eqref{eq:left} tensored with $\alpha\beta$; similarly for \eqref{eq:right}.

(2) Injectivity is proved as in Lemma~\ref{lem:homog}. Let $f\in\Hom_\Omega(A^{\ot n},M)$, $X_i=a_i\ot\alpha_i$ and $\omega=\alpha_1\cdots\alpha_{n+1}$. By Proposition~\ref{prop:Omega-assoc}, the deformed product of $\widetilde N$ is $X_i\ast X_{i+1}=(a_i\cdot_{N;\alpha_i,\alpha_{i+1}}a_{i+1})\ot\alpha_i\alpha_{i+1}$, and

\begin{align*}
X_1\cdot\iota(f)(X_2,\ldots,X_{n+1})&=\big(a_1\cdot f_{\alpha_2,\ldots,\alpha_{n+1}}(a_2,\ldots,a_{n+1})\big)\ot\omega,\\
\widetilde N(X_1)\cdot\iota(f)(X_2,\ldots,X_{n+1})&=\big(N_{\alpha_1}(a_1)\cdot f_{\alpha_2,\ldots,\alpha_{n+1}}(a_2,\ldots,a_{n+1})\big)\ot\omega,\\
\iota(f)(\ldots,X_iX_{i+1},\ldots)&=f_{\boldsymbol\alpha^{(i)}}(\ldots,a_i\cdot a_{i+1},\ldots)\ot\omega,
\end{align*}
and similarly for the right actions of $X_{n+1}$ and $\widetilde N(X_{n+1})$ and for $X_i\ast X_{i+1}$. Hence $\delta^n_{\Hoch,M_\Omega}\circ\iota=\iota\circ\delta^n_{\Hoch,M}$ and $\partial^n\circ\iota=\iota\circ\partial$, where $\partial^n$ is the operator of \cite[Section~2.2]{song} for $(A_\Omega,\widetilde N)$. Since moreover $\widetilde N_M\circ\iota(g)=\iota(N_M\circ g)$, \eqref{eq:partial-short} and \cite[(7)]{song} give $\iota\circ\delta^n_{\NF,M}=\delta^n_{\NF,M_\Omega}\circ\iota$. Finally, since $\widetilde N$ does not change the indices,
\begin{align*}
\Phi^n_{M_\Omega}\big(\iota(f)\big)(X_1,\ldots,X_n)&=\sum_{S\subseteq\{1,\ldots,n\}}\big(-N_{M,\,\alpha_1\cdots\alpha_n}\big)^{n-|S|}f_{\alpha_1,\ldots,\alpha_n}(\mathbf a^S)\ot\alpha_1\cdots\alpha_n\\
&=\iota\big(\Phi^n_M(f)\big)(X_1,\ldots,X_n).
\end{align*}
If $\Omega$ has a unit $1$, then $\delta^0_{\NF,M_\Omega}(m\ot1)=\iota\big(\delta^0_{\NF,M}(m)\big)$ by \eqref{eq:delta0-NF}, and similarly for $\delta^0_{\Hoch,M}$ and $\Phi^0_M=\id_M$.

(3) For $n\geq2$, $C^n_\NFa(A,M)=\Hom(A^{\ot n},M)\oplus\Hom_\Omega(A^{\ot n-1},M)$ with $n-1\geq1$, so (3) follows from (2) and from the definition of the differentials as mapping cones.
\end{proof}

Proposition~\ref{prop:song-embedding} explains the main difference with the theory of the Nijenhuis operator $\widetilde N$ on $A_\Omega$: in $C^\bullet_\NFa(A,M)$, the product of $A$ is deformed by constant collections, while the operators $N_\omega$ are deformed by arbitrary collections, so that $\rmH^\bullet_\NFa(A,M)$ is the cohomology of the subcomplex of Proposition~\ref{prop:song-embedding}(3). For $\Omega$ trivial, $C^\bullet_\NFa(A)$ is the complex of \cite[Definition~2.5]{song}.

\section{Deformations of Nijenhuis family algebras}\label{sec:deformations}

In this section, we study formal and infinitesimal deformations of a Nijenhuis family algebra $(A,\{N_\omega\}_{\omega\in\Omega})$, in which both the product $\mu$ of $A$ and the family $\{N_\omega\}$ vary. We show that equivalence classes of infinitesimal deformations are in bijection with $\rmH^2_\NFa(A)$, the cohomology with coefficients in the regular bimodule.

\begin{defn}\label{def:formal}
A \emph{formal one-parameter deformation} of $(A,\{N_\omega\})$ consists of formal power series $\mu_t=\sum_{i\geq0}t^i\mu_i$ and $N_{t,\omega}=\sum_{i\geq0}t^iN^{(i)}_{\omega}$, where $\mu_i\in\Hom(A^{\ot2},A)$, $N^{(i)}_{\omega}\in\Hom(A,A)$, $\mu_0=\mu$ and $N^{(0)}_{\omega}=N_\omega$, such that $\big((A[\![t]\!],\mu_t),\{N_{t,\omega}\}\big)$ is a Nijenhuis family algebra over $\bk[\![t]\!]$.
\end{defn}

Thus $\big(\mu_t,\{N_{t,\omega}\}\big)$ is a formal deformation if and only if, for all $n\geq0$, $a,b,c\in A$ and $\alpha,\beta\in\Omega$,
\begin{align}
\label{Eq: deform eq for products}\sum_{i+j=n}\mu_i\big(\mu_j(a,b),c\big)&=\sum_{i+j=n}\mu_i\big(a,\mu_j(b,c)\big),
\end{align}
\begin{multline}\label{Eq: Deform N operator}
\sum_{i+j+k=n}\mu_i\big(N^{(j)}_{\alpha}(a),N^{(k)}_{\beta}(b)\big)=\sum_{i+j+k=n}N^{(i)}_{\alpha\beta}\Big(\mu_j\big(N^{(k)}_{\alpha}(a),b\big)+\mu_j\big(a,N^{(k)}_{\beta}(b)\big)\Big)\\-\sum_{i+j+k=n}N^{(i)}_{\alpha\beta}N^{(j)}_{\alpha\beta}\,\mu_k(a,b).
\end{multline}
For $n=0$, these are the identities of $(A,\{N_\omega\})$. For $n=1$, we obtain
\begin{equation}\label{eq:d19}
\mu_1(a\cdot b,c)+\mu_1(a,b)\cdot c=\mu_1(a,b\cdot c)+a\cdot\mu_1(b,c),
\end{equation}
\begin{multline}\label{eq:d23}
\mu_1\big(N_\alpha(a),N_\beta(b)\big)+N^{(1)}_{\alpha}(a)\cdot N_\beta(b)+N_\alpha(a)\cdot N^{(1)}_{\beta}(b)\\
=N^{(1)}_{\alpha\beta}\big(N_\alpha(a)\cdot b+a\cdot N_\beta(b)\big)+N_{\alpha\beta}\Big(\mu_1\big(N_\alpha(a),b\big)+\mu_1\big(a,N_\beta(b)\big)+N^{(1)}_{\alpha}(a)\cdot b+a\cdot N^{(1)}_{\beta}(b)\Big)\\
-N^{(1)}_{\alpha\beta}N_{\alpha\beta}(a\cdot b)-N_{\alpha\beta}N^{(1)}_{\alpha\beta}(a\cdot b)-N_{\alpha\beta}^2\,\mu_1(a,b).
\end{multline}
The identity \eqref{eq:d19} means $\delta^2_\Hoch(\mu_1)=0$. In \eqref{eq:d23}, we have
\[N^{(1)}_{\alpha\beta}\big(N_\alpha(a)\cdot b+a\cdot N_\beta(b)\big)-N^{(1)}_{\alpha\beta}N_{\alpha\beta}(a\cdot b)=N^{(1)}_{\alpha\beta}(a\cdot_{N;\alpha,\beta}b).\]
Collecting the terms containing $\mu_1$ and those containing $N^{(1)}=\{N^{(1)}_{\omega}\}$, \eqref{eq:d23} becomes
\begin{multline*}
\Big[\mu_1\big(N_\alpha(a),N_\beta(b)\big)-N_{\alpha\beta}\mu_1\big(N_\alpha(a),b\big)-N_{\alpha\beta}\mu_1\big(a,N_\beta(b)\big)+N^2_{\alpha\beta}\mu_1(a,b)\Big]\\
+\Big[N_\alpha(a)\cdot N^{(1)}_{\beta}(b)-N_{\alpha\beta}\big(a\cdot N^{(1)}_{\beta}(b)\big)-N^{(1)}_{\alpha\beta}(a\cdot_{N;\alpha,\,\beta}b)+N_{\alpha\beta}N^{(1)}_{\alpha\beta}(a\cdot b)\\
+N^{(1)}_{\alpha}(a)\cdot N_\beta(b)-N_{\alpha\beta}\big(N^{(1)}_{\alpha}(a)\cdot b\big)\Big]=0.
\end{multline*}
By \eqref{eq:Phi}, the first bracket is $\Phi^2(\mu_1)_{\alpha,\,\beta}(a,b)$, and, by the formula of $\delta^1_\NF$, the second one is $\delta^1_\NF(N^{(1)})_{\alpha,\,\beta}(a,b)$. Hence \eqref{eq:d19} and \eqref{eq:d23} are equivalent to
\[\delta^2_{\NFa}(\mu_1,N^{(1)})=\Big(\delta^2_\Hoch(\mu_1),\ -\delta^1_\NF(N^{(1)})-\Phi^2(\mu_1)\Big)=0,\]
that is, $(\mu_1,N^{(1)})\in C^2_\NFa(A)=\Hom(A^{\ot2},A)\oplus\Hom_\Omega(A,A)$ is a $2$-cocycle (see \eqref{eq:2-cocycle} with $M=A$). It is called the \emph{infinitesimal} of the deformation.

\begin{defn}\label{def:equivalence}
Two formal deformations $\big(\mu_t,\{N_{t,\omega}\}\big)$ and $\big(\mu'_t,\{N'_{t,\omega}\}\big)$ of $(A,\{N_\omega\})$ are \emph{equivalent} if there is a formal power series $\varphi_t=\sum_{i\geq0}t^i\varphi_i$, with $\varphi_i\in\Hom(A,A)$ and $\varphi_0=\id_A$, such that $\varphi_t:\big((A[\![t]\!],\mu_t),\{N_{t,\omega}\}\big)\to\big((A[\![t]\!],\mu'_t),\{N'_{t,\omega}\}\big)$ is a morphism of Nijenhuis family algebras.
\end{defn}

This means that, for all $n\geq0$, $a,b\in A$ and $\omega\in\Omega$,
\begin{equation}\label{Eq: equivalent deformations}
\sum_{i+j=n}\varphi_i\big(\mu_j(a,b)\big)=\sum_{i+j+k=n}\mu'_i\big(\varphi_j(a),\varphi_k(b)\big),\qquad\sum_{i+j=n}\varphi_i\circ N^{(j)}_{\omega}=\sum_{i+j=n}N'^{(i)}_{\omega}\circ\varphi_j.
\end{equation}
For $n=1$, the first identity gives $\varphi_1(a\cdot b)+\mu_1(a,b)=\mu'_1(a,b)+\varphi_1(a)\cdot b+a\cdot\varphi_1(b)$, that is, $\mu_1-\mu'_1=\delta^1_\Hoch(\varphi_1)$, and the second one gives $\varphi_1\circ N_\omega+N^{(1)}_{\omega}=N'^{(1)}_{\omega}+N_\omega\circ\varphi_1$, that is,
\[N^{(1)}_{\omega}-N'^{(1)}_{\omega}=N_\omega\circ\varphi_1-\varphi_1\circ N_\omega=-\Phi^1(\varphi_1)_\omega.\]
Hence
\[(\mu_1,N^{(1)})-(\mu'_1,N'^{(1)})=\big(\delta^1_\Hoch(\varphi_1),-\Phi^1(\varphi_1)\big)=\delta^1_{\NFa}(\varphi_1,0).\]

In conclusion, we have the following result.

\begin{thm}\label{thm: equivalent deformations}
The infinitesimal of a formal deformation of $(A,\{N_\omega\})$ is a $2$-cocycle in $C^\bullet_\NFa(A)$, and its class in $\rmH^2_\NFa(A)$ depends only on the equivalence class of the deformation.
\end{thm}

It follows from Theorem~\ref{thm: equivalent deformations} that an equivalence class of formal deformations gives rise to a class in $\rmH^2_\NFa(A)$; the converse need not be true. We now consider infinitesimal deformations, for which this map becomes bijective.

\begin{defn}
An \emph{infinitesimal deformation} of $(A,\{N_\omega\})$ consists of $\big(\mu_t=\mu+t\mu_1,\{N_{t,\omega}=N_\omega+tN^{(1)}_{\omega}\}\big)$ such that $\big((A[t]/(t^2),\mu_t),\{N_{t,\omega}\}\big)$ is a Nijenhuis family algebra over $\bk[t]/(t^2)$. Equivalences are defined as in Definition~\ref{def:equivalence}, with $\varphi_t=\id_A+t\varphi_1$.
\end{defn}

\begin{thm}\label{thm:infinitesimal}
There is a bijection between the set of equivalence classes of infinitesimal deformations of $(A,\{N_\omega\})$ and $\rmH^2_\NFa(A)$.
\end{thm}

\begin{proof}
Over $\bk[t]/(t^2)$, only the identities \eqref{Eq: deform eq for products} and \eqref{Eq: Deform N operator} for $n=0,1$ remain. By the above computation, $\big(\mu+t\mu_1,\{N_\omega+tN^{(1)}_{\omega}\}\big)$ is an infinitesimal deformation if and only if $(\mu_1,N^{(1)})$ is a $2$-cocycle, and two infinitesimal deformations are equivalent via $\varphi_t=\id_A+t\varphi_1$ if and only if the difference of their cocycles is $\delta^1_{\NFa}(\varphi_1,0)$. It remains to show that every $2$-coboundary has this form. Let $(\varphi,x)\in C^1_\NFa(A)=\Hom(A,A)\oplus A$. Since $\delta^1_{\NFa}\circ \delta^0_{\NFa}=0$ and $\delta^0_{\NFa}(x)=\big(\delta^0_\Hoch(x),-x\big)$, we have $\delta^1_{\NFa}(0,x)=\delta^1_{\NFa}\big(\delta^0_\Hoch(x),0\big)$, so $\delta^1_{\NFa}(\varphi,x)=\delta^1_{\NFa}\big(\varphi+\delta^0_\Hoch(x),0\big)$. Hence the map $[(\mu_1,N^{(1)})]\mapsto[\text{deformation}]$ is well defined and bijective.
\end{proof}

A formal deformation is \emph{trivial} if it is equivalent to the deformation $\big(\mu,\{N_\omega\}\big)$, with $\mu_i=0$ and $N^{(i)}_{\omega}=0$ for $i\geq1$.

\begin{thm}\label{thm:rigidity}
If $\rmH^2_\NFa(A)=0$, then every formal deformation of $(A,\{N_\omega\})$ is trivial.
\end{thm}

\begin{proof}
Let $\big(\mu_t,\{N_{t,\omega}\}\big)$ be a formal deformation, and assume that $\mu_i=0$ and $N^{(i)}_{\omega}=0$ for $1\leq i<n$, for some $n\geq1$. In \eqref{Eq: deform eq for products} and \eqref{Eq: Deform N operator} at order $n$, the only nonzero terms are those in which one index equals $n$ and the others are $0$. Hence these equations are \eqref{eq:d19} and \eqref{eq:d23} with $(\mu_1,N^{(1)})$ replaced by $(\mu_n,N^{(n)})$, so $(\mu_n,N^{(n)})$ is a $2$-cocycle. Since $\rmH^2_\NFa(A)=0$, there is $(\varphi,x)\in C^1_\NFa(A)$ with $(\mu_n,N^{(n)})=\delta^1_\NFa(\varphi,x)=\delta^1_\NFa(\varphi',0)$, where $\varphi'=\varphi+\delta^0_\Hoch(x)$, as in the proof of Theorem~\ref{thm:infinitesimal}.

Put $\varphi_t:=\id_A+t^n\varphi'$, invertible over $\bk[\![t]\!]$, and transport the deformation along $\varphi_t$:
\[\mu'_t(a,b):=\varphi_t\big(\mu_t(\varphi_t^{-1}a,\varphi_t^{-1}b)\big),\qquad N'_{t,\omega}:=\varphi_t\circ N_{t,\omega}\circ\varphi_t^{-1}.\]
As $\varphi_t\equiv\id_A$ modulo $t^n$, we have $\mu'_i=0$, $N'^{(i)}_{\omega}=0$ for $1\leq i<n$, and at order $n$, \eqref{Eq: equivalent deformations} gives
\[(\mu_n,N^{(n)})-(\mu'_n,N'^{(n)})=\big(\delta^1_\Hoch(\varphi'),-\Phi^1(\varphi')\big)=\delta^1_\NFa(\varphi',0)=(\mu_n,N^{(n)}),\]
so $\mu'_n=0$ and $N'^{(n)}=0$. By induction we get equivalences $\varphi^{(n)}_t\equiv\id_A$ modulo $t^n$; their infinite composition converges $t$-adically and is an equivalence onto the trivial deformation.
\end{proof}

\begin{remark}\label{rem:obstructions}
Let $\big(\mu_t,\{N_{t,\omega}\}\big)$ satisfy \eqref{Eq: deform eq for products} and \eqref{Eq: Deform N operator} up to order $n$. By the same computation as for $n=1$, these identities at order $n+1$ read $\delta^2_\NFa(\mu_{n+1},N^{(n+1)})=\mathrm{Ob}_n$, where $\mathrm{Ob}_n\in C^3_\NFa(A)$ collects the terms of order $n+1$ involving only $\mu_i$ and $N^{(i)}_{\omega}$ with $1\leq i\leq n$. As in \cite{gers}, one expects $\mathrm{Ob}_n$ to be a $3$-cocycle, whose class is the obstruction to extending the deformation; this can be deduced from the Maurer--Cartan equation of Theorem~\ref{thm:MC-NFalg}, and we do not detail it here.
\end{remark}

\section{Homotopy Nijenhuis family algebras and \texorpdfstring{$A_\infty$}{A-infinity}-algebras}\label{sec:homotopy}

In this section, we introduce homotopy Nijenhuis family algebras and study their relation with $A_\infty$-algebras \cite{Stasheff}, which underlie them. We use the homological grading and the sign conventions of \cite{song}; in particular, in this section, the maps $l_k$ of an $L_\infty$-algebra have degree $k-2$ and Maurer--Cartan elements have degree $-1$, which differs from the conventions of Section~\ref{sec:background} by a shift of degrees: for $A$ concentrated in degree $0$, an element of $\Hom\big((sA)^{\ot n},sA\big)$ or of $\Hom_\Omega\big((sA)^{\ot n-1},A\big)$ has degree $1-n$ here, and degree $n-2$ in $s^{-1}\mathfrak h_A\oplus\mathfrak a_A$ (Theorem~\ref{thm:Linfty}).

\begin{defn}\label{def:Ainfty}
An \emph{$A_\infty$-algebra} is a pair $(A,\{m_n\}_{n\geq1})$ consisting of a graded vector space $A$ and a collection of linear maps $m_n:A^{\ot n}\to A$ of degree $n-2$ satisfying, for all $n\geq1$,
\begin{equation}\label{eq:Stasheff}
\sum_{\substack{i+j+k=n\\ i,k\geq0,\ j\geq1}}(-1)^{i+jk}\,m_{i+1+k}\circ\big(\id^{\ot i}\ot m_j\ot\id^{\ot k}\big)=0.
\end{equation}
\end{defn}

\subsection{Homotopy Nijenhuis family algebras}
We generalize the results of Section~\ref{subsec:Linfty} in the homotopy context. Let $A$ be a graded vector space and $sA$ its suspension, with $s:A\to sA$ of degree $1$; by the Koszul rule, $(s^{-1})^{\ot n}\circ s^{\ot n}=(-1)^{n(n-1)/2}\id$. Via
\begin{equation}\label{eq:shift-iso}
m_n=s^{-1}\circ b_n\circ s^{\ot n},\qquad N_n=R_n\circ s^{\ot n},
\end{equation}
maps $b_n:(sA)^{\ot n}\to sA$ of degree $-1$ correspond to maps $m_n:A^{\ot n}\to A$ of degree $n-2$, and collections $R_n\in\Hom_\Omega\big((sA)^{\ot n},A\big)$ of degree $-1$ correspond to collections $N_n\in\Hom_\Omega(A^{\ot n},A)$ of degree $n-1$ \cite[(19)--(20)]{song}. In the notation $N_n=\{N_{n;\alpha_1,\ldots,\alpha_n}\}$, the integer before the semicolon is the arity and the elements of $\Omega$ after it are the indices. For a graded vector space $V$ and $f,g_1,\ldots,g_k\in\mathrm{End}^\Omega_V$, the \emph{brace} $f\{g_1,\ldots,g_k\}$ is the sum, with Koszul signs, of all the insertions of $g_1,\ldots,g_k$, in this order, into $k$ of the inputs of $f$ \cite[Definition~3.6]{CGWZ}, computed in $\mathrm{End}^\Omega_V$.

Put $C_\Hoch(A):=\Hom\big(T^c(sA),sA\big)$ and $C^\Omega_\NF(A):=\Hom_\Omega\big(T^c(sA),A\big)$, where $T^c(sA)=\bigoplus_{n\geq1}(sA)^{\ot n}$. For $\Omega$ trivial, \cite[Definition~5.1]{song} endows $C_\Hoch(A)\oplus\Hom\big(T^c(sA),A\big)$ with an $L_\infty$-algebra structure $\{l_k\}_{k\geq1}$ given by brace formulas. By Lemma~\ref{lem:transport}, as in the proof of Theorem~\ref{thm:Linfty}, the same formulas define an $L_\infty$-algebra structure on $C_\Hoch(A)\oplus C^\Omega_\NF(A)$. An element of degree $-1$ of this space is a pair $\big(\{b_n\},\{R_n\}\big)$ as above.

\begin{defn}\label{def:homotopy}
A \emph{homotopy Nijenhuis family algebra} is a triple $\big(A,\{m_n\}_{n\geq1},\{N_n\}_{n\geq1}\big)$, where $A$ is a graded vector space, $m_n:A^{\ot n}\to A$ are linear maps of degree $n-2$ and $N_n\in\Hom_\Omega(A^{\ot n},A)$ are collections of degree $n-1$, such that the corresponding element $\big(\{b_n\},\{R_n\}\big)$ given by \eqref{eq:shift-iso} is a Maurer--Cartan element of the $L_\infty$-algebra $\big(C_\Hoch(A)\oplus C^\Omega_\NF(A),\{l_k\}\big)$.
\end{defn}

For $\Omega$ trivial, Definition~\ref{def:homotopy} is \cite[Definition~5.2]{song}. The equivalent explicit description of a homotopy Nijenhuis family algebra is given by the following result.

\begin{thm}\label{thm:homotopy-explicit}
A triple $\big(A,\{m_n\},\{N_n\}\big)$ is a homotopy Nijenhuis family algebra if and only if, for all $n\geq1$, the corresponding maps \eqref{eq:shift-iso} satisfy
\begin{align}
\label{eq:MC-Ainfty}&\sum_{i=1}^nb_{n-i+1}\{b_i\}=0,\\
\label{eq:MC-homotopy}&\sum_{\substack{r_1+\cdots+r_p=n\\ r_1,\ldots,r_p\geq1,\ p\geq1}}\ \sum_{t=0}^p(-1)^t\,(sR_{r_1})\Big\{\cdots\Big\{sR_{r_t}\big\{b_p\{sR_{r_{t+1}},\ldots,sR_{r_p}\}\big\}\Big\}\cdots\Big\}=0
\end{align}
in $\mathrm{End}^\Omega_{sA}$, where $sR_r:=s\circ R_r$. In particular, $(A,\{m_n\})$ is an $A_\infty$-algebra.
\end{thm}

\begin{proof}
For $\Omega$ trivial, the first assertion is \cite[Definition~5.2, (17)--(18)]{song}: the component in $C_\Hoch(A)$ of the Maurer--Cartan equation is the left hand side of \eqref{eq:MC-Ainfty}, and its component in $\Hom\big(T^c(sA),A\big)$ is the left hand side of \eqref{eq:MC-homotopy}. All the expressions involved are as in Lemma~\ref{lem:transport}, so the same holds for collections. Finally, \eqref{eq:MC-Ainfty} is equivalent to \eqref{eq:Stasheff} via \eqref{eq:shift-iso}.
\end{proof}

\begin{remark}\label{rem:operad}
Equation \eqref{eq:MC-Ainfty} means that $(sA,\{b_n\})$ is an $A_\infty[1]$-algebra. Recall \cite[Definition~4.1, Theorem~4.4]{song} that the minimal model $\mathrm{NjA}_\infty$ of the operad of Nijenhuis associative algebras is the free graded nonsymmetric operad generated by $m_n$ ($n\geq2$) of degree $n-2$ and $N_n$ ($n\geq1$) of degree $n-1$, with $\partial(N_1)=0$ and
\begin{align}
\label{eq:dm3}\partial(m_3)&=-m_2\circ_1m_2+m_2\circ_2m_2,\\
\label{eq:dN2}\partial(N_2)&=-(m_2\circ_1N_1)\circ_2N_1+N_1\circ_1(m_2\circ_1N_1)+N_1\circ_1(m_2\circ_2N_1)-(N_1\circ_1N_1)\circ_1m_2
\end{align}
in low arities. The operad $\mathrm{End}^\Omega_A$ is a dg operad for $\partial(f):=m_1\circ f-(-1)^{|f|}\sum_if\circ_im_1$. Via \eqref{eq:shift-iso}, \eqref{eq:Stasheff} and \eqref{eq:MC-homotopy} are equivalent to the identities \cite[(13)--(14)]{song} in $\mathrm{End}^\Omega_A$, by \cite[Proposition~3.9]{CGWZ} and Lemma~\ref{lem:transport}, that is, a homotopy Nijenhuis family algebra structure on $A$ is the same as a morphism of dg operads $\phi:\mathrm{NjA}_\infty\to\mathrm{End}^\Omega_A$ sending each $m_n$ into the constant collections.
\end{remark}

In low arities, we obtain the following result, which is the family version of \cite[(15)--(16)]{song}: the family $N_1$ is a Nijenhuis family with respect to $m_2$ up to homotopy, the obstruction being given by $N_2$.

\begin{thm}\label{thm:homotopy-low}
Let $\big(A,\{m_n\},\{N_n\}\big)$ be a homotopy Nijenhuis family algebra. Then, for all $\alpha,\beta\in\Omega$,
\begin{align}
\label{eq:h1}&m_1\circ N_{1;\alpha}=N_{1;\alpha}\circ m_1,\\
\label{eq:h2}&m_2\circ(N_{1;\alpha}\ot N_{1;\beta})-N_{1;\alpha\beta}\circ m_2\circ(N_{1;\alpha}\ot\id)-N_{1;\alpha\beta}\circ m_2\circ(\id\ot N_{1;\beta})+N_{1;\alpha\beta}\circ N_{1;\alpha\beta}\circ m_2\notag\\
&\qquad=-\big(m_1\circ N_{2;\alpha,\beta}+N_{2;\alpha,\beta}\circ(m_1\ot\id)+N_{2;\alpha,\beta}\circ(\id\ot m_1)\big).
\end{align}
Consequently, writing $N_\omega:=N_{1;\omega}$, of degree $0$, and still denoting by $N_\omega$ the induced maps on homology, $\big(\rmH_\bullet(A,m_1),\{N_\omega\}\big)$, with the product induced by $m_2$, is a graded Nijenhuis family algebra.
\end{thm}

\begin{proof}
Let $\phi:\mathrm{NjA}_\infty\to\mathrm{End}^\Omega_A$ be the morphism of dg operads of Remark~\ref{rem:operad}. Then $\partial\big(\phi(x)\big)=\phi\big(\partial(x)\big)$ for every generator $x$. Since $N_1$ has degree $0$ and $\partial(N_1)=0$, we get $m_1\circ N_1-N_1\circ m_1=0$ in $\mathrm{End}^\Omega_A$, which is \eqref{eq:h1}. Since $N_2$ has degree $1$, and by the Koszul rule $N_2\circ_1m_1=N_2\circ(m_1\ot\id)$ and $N_2\circ_2m_1=N_2\circ(\id\ot m_1)$,
\[\partial\big(\phi(N_2)\big)_{\alpha,\beta}=m_1\circ N_{2;\alpha,\beta}+N_{2;\alpha,\beta}\circ(m_1\ot\id)+N_{2;\alpha,\beta}\circ(\id\ot m_1).\]
On the other hand, by the rule for the indices, the components at $(\alpha,\beta)$ of the four terms of \eqref{eq:dN2} are

\begin{align*}
\big((m_2\circ_1N_1)\circ_2N_1\big)_{\alpha,\beta}&=m_2\circ(N_{1;\alpha}\ot N_{1;\beta}),\\
\big(N_1\circ_1(m_2\circ_1N_1)\big)_{\alpha,\beta}&=N_{1;\alpha\beta}\circ m_2\circ(N_{1;\alpha}\ot\id),\\
\big(N_1\circ_1(m_2\circ_2N_1)\big)_{\alpha,\beta}&=N_{1;\alpha\beta}\circ m_2\circ(\id\ot N_{1;\beta}),\\
\big((N_1\circ_1N_1)\circ_1m_2\big)_{\alpha,\beta}&=N_{1;\alpha\beta}\circ N_{1;\alpha\beta}\circ m_2,
\end{align*}
since the output of $m_2$ carries the index $\alpha\beta$. Hence $\phi\big(\partial(N_2)\big)=\partial\big(\phi(N_2)\big)$ is exactly \eqref{eq:h2}.

\emph{Homology.} By \eqref{eq:h1}, each $N_\omega=N_{1;\omega}$ sends cycles to cycles and boundaries to boundaries, so it induces a map on $\rmH_\bullet(A,m_1)$. By the identities \eqref{eq:Stasheff} of $(A,\{m_n\})$ for $n=2$ and $n=3$, $m_2$ satisfies the Leibniz rule with respect to $m_1$ and is associative up to the homotopy $m_3$, so it induces an associative product on $\rmH_\bullet(A,m_1)$. Let $x,y$ be $m_1$-cycles. Then $(m_1\ot\id)(x\ot y)=0$ and $(\id\ot{m}_1)(x\ot y)=(-1)^{|x|}x\ot{m}_1(y)=0$, so \eqref{eq:h2} evaluated on $x\ot y$ gives
\[m_2\big(N_\alpha x,N_\beta y\big)-N_{\alpha\beta}\Big(m_2(N_\alpha x,y)+m_2(x,N_\beta y)-N_{\alpha\beta}\,m_2(x,y)\Big)=-m_1\big(N_{2;\alpha,\beta}(x,y)\big).\]
The right hand side is a boundary, so $\{N_\omega\}$ induces a Nijenhuis family on $\rmH_\bullet(A,m_1)$.
\end{proof}

The following result, which is the family version of \cite[Proposition~5.4]{song}, shows that any Nijenhuis family algebra can be seen as a homotopy Nijenhuis family algebra concentrated in degree $0$.

\begin{prop}\label{prop:homotopy-deg0}
Let $A$ be a vector space, regarded as a graded space concentrated in degree $0$, $\mu:A\ot A\to A$ a linear map and $N=\{N_\omega\}_{\omega\in\Omega}$ a collection of linear maps on $A$. Then $\big((A,\mu),\{N_\omega\}\big)$ is a Nijenhuis family algebra if and only if
\[b_2:=-s\circ\mu\circ(s^{-1})^{\ot 2},\qquad R_1:=\{N_\omega\circ s^{-1}\}_{\omega\in\Omega},\qquad b_n=0\ (n\neq2),\qquad R_n=0\ (n\geq2),\]
is a Maurer--Cartan element of $C_\Hoch(A)\oplus C^\Omega_\NF(A)$. Moreover, every homotopy Nijenhuis family algebra structure on $A$ is of this form.
\end{prop}

\begin{proof}
Since $sA$ is concentrated in degree $1$, $(sA)^{\ot n}$ is concentrated in degree $n$. Hence a map $b_n:(sA)^{\ot n}\to sA$ of degree $-1$ vanishes unless $n-1=1$, and a map $R_n:(sA)^{\ot n}\to A$ of degree $-1$ vanishes unless $n-1=0$. This proves the last assertion. Recall that, by the Koszul rule, $(s^{-1})^{\ot n}\circ s^{\ot n}=(-1)^{n(n-1)/2}\id$; in particular, $(s^{-1})^{\ot2}\circ s^{\ot2}=(-1)^{|s^{-1}||s|}\,s^{-1}s\ot s^{-1}s=-\id$. By \eqref{eq:shift-iso}, we get $m_2=s^{-1}b_2\,s^{\ot 2}=\mu$ and $N_{1;\omega}=N_\omega$, while $m_n=0$ for $n\neq2$ and $N_n=0$ for $n\geq2$. Then \eqref{eq:Stasheff} reduces, for $n=3$, to $m_2\circ(m_2\ot\id)-m_2\circ(\id\ot m_2)=0$, that is, to the associativity of $\mu$, and, by Theorems~\ref{thm:homotopy-explicit} and~\ref{thm:homotopy-low}, \eqref{eq:MC-homotopy} reduces for $n=2$ to \eqref{eq:h2} with $m_1=0$ and $N_2=0$, that is, to \eqref{eq:NF-intro}; the other identities are trivially satisfied.
\end{proof}

This Maurer--Cartan element recovers the complex of Section~\ref{sec:cohomology}. For $n\geq2$ and $(f,g)\in C^n_\NFa(A)=\Hom(A^{\ot n},A)\oplus\Hom_\Omega(A^{\ot n-1},A)$, let $\check f\in\Hom\big((sA)^{\ot n},sA\big)$ and $\hat g\in\Hom_\Omega\big((sA)^{\ot n-1},A\big)$ be the elements corresponding to $f$ and $g$ under \eqref{eq:shift-iso}, that is, $s^{-1}\check f\,s^{\ot n}=f$ and $\hat g\,s^{\ot n-1}=g$. Both have degree $1-n$.

\begin{prop}\label{prop:twisted}
Let $(A,\{N_\omega\})$ be a Nijenhuis family algebra, and let $\theta=(b_2,R_1)$ be the corresponding Maurer--Cartan element of Proposition~\ref{prop:homotopy-deg0}. Let $l^\theta_1$ be the differential of the $L_\infty$-algebra $C_\Hoch(A)\oplus C^\Omega_\NF(A)$ twisted by $\theta$. Then, for $n\geq2$ and $(f,g)\in C^n_\NFa(A)$, the element $l^\theta_1(\check f,\hat g)$ corresponds under \eqref{eq:shift-iso} to
\[L^n(f,g):=\Big((-1)^n\delta^n_\Hoch(f),\ \Phi^n(f)+(-1)^n\delta^{n-1}_\NF(g)\Big)\in C^{n+1}_\NFa(A).\]
Consequently, the maps
\[\psi_n:C^n_\NFa(A)\to C^n_\NFa(A),\qquad \psi_n(f,g):=\Big((-1)^{\frac{n(n-1)}2}f,\ -(-1)^{\frac{(n-1)(n-2)}2}g\Big),\qquad n\geq2,\]
satisfy $\psi_{n+1}\circ\delta^n_\NFa=L^n\circ\psi_n$, and they identify $\big(C^{\geq2}_\NFa(A),\delta_\NFa\big)$ with the twisted complex $\big(C_\Hoch(A)\oplus C^\Omega_\NF(A),l^\theta_1\big)$ in arities $\geq2$, up to the shift of degrees.
\end{prop}

\begin{proof}
For $\Omega$ trivial, the first assertion is the computation in the proof of \cite[Proposition~5.5]{song}: with $\nu=b_2$ and $\tau=R_1$, the term $-[\nu,\check f]_G$ corresponds to $(-1)^n\delta^n_\Hoch(f)$, the term $(-1)^{\frac{n(n+1)}2}\frac1{n!}\,l_{n+1}(\tau^{\ot n}\ot\check f)$ corresponds to $\Phi^n(f)$, and $l^\theta_1(\hat g)=-l_3(\nu\ot\tau\ot\hat g)$ corresponds to $(-1)^n\delta^{n-1}_\NF(g)$. All these terms are expressions as in Lemma~\ref{lem:transport}. Applying them to $A_\Omega$ with the Maurer--Cartan element $\big(\iota(b_2),\iota(R_1)\big)$ of the Nijenhuis associative algebra $(A_\Omega,\widetilde N)$, and using Lemma~\ref{lem:transport} and Proposition~\ref{prop:song-embedding}(2), we obtain the first assertion for an arbitrary $\Omega$.

Write $a_n:=(-1)^{\frac{n(n-1)}2}$ and $c_n:=-a_{n-1}$, so that $\psi_n(f,g)=(a_nf,c_ng)$, $a_{n+1}=(-1)^na_n$ and $c_{n+1}=-a_n=(-1)^nc_n$. By Definition~\ref{def:cone},
$\psi_{n+1}\big(\delta^n_\NFa(f,g)\big)=\big(a_{n+1}\,\delta^n_\Hoch(f),\ -c_{n+1}(\delta^{n-1}_\NF(g)+\Phi^n(f))\big)=\big((-1)^na_n\,\delta^n_\Hoch(f),\ a_n\,\Phi^n(f)+(-1)^nc_n\,\delta^{n-1}_\NF(g)\big)$, which is $L^n(a_nf,c_ng)=L^n\big(\psi_n(f,g)\big)$.
Since the maps $\psi_n$ are bijective, this proves the last assertion.
\end{proof}

For $n=2$, Proposition~\ref{prop:twisted} can be checked independently: the kernel of $l^\theta_1$ in this degree consists of the infinitesimal deformations, which, by Section~\ref{sec:deformations}, are the $2$-cocycles $(\mu_1,N^{(1)})$, that is, $\delta^2_\Hoch(\mu_1)=0$ and $\delta^1_\NF(N^{(1)})+\Phi^2(\mu_1)=0$, which is the kernel of $L^2$.

A homotopy Nijenhuis family algebra $\big(A,\{m_n\},\{N_n\}\big)$ is said to be \emph{strict} if $N_n=0$ for $n\geq2$. In this case, $N$ is given by a collection $\{N_\omega:=N_{1;\omega}\}_{\omega\in\Omega}$ of linear maps of degree $0$, and it follows from \eqref{eq:MC-homotopy} that this collection satisfies the following identities.

\begin{prop}\label{prop:strict}
Let $(A,\{m_n\})$ be an $A_\infty$-algebra and $N=\{N_\omega\}_{\omega\in\Omega}$ a collection of linear maps of degree $0$ on $A$. Then $N$ defines a strict homotopy Nijenhuis family algebra on $A$, that is, $N_{1;\omega}=N_\omega$ for $\omega\in\Omega$ and $N_n=0$ for $n\geq2$, if and only if, for all $n\geq1$ and $\alpha_1,\ldots,\alpha_n\in\Omega$,
\begin{equation}\label{eq:strict}
\Phi^n_N(m_n)_{\alpha_1,\ldots,\alpha_n}:=\sum_{S\subseteq\{1,\ldots,n\}}\big(-N_{\alpha_1\cdots\alpha_n}\big)^{n-|S|}\circ m_n\circ\big(N^S_{\alpha_1}\ot\cdots\ot N^S_{\alpha_n}\big)=0,
\end{equation}
where $N^S_{\alpha_i}=N_{\alpha_i}$ if $i\in S$ and $N^S_{\alpha_i}=\id$ otherwise. In other words, $\Phi^n_N$ is the map \eqref{eq:Phi} for the regular bimodule, with respect to the family $N$.
\end{prop}

\begin{proof}
By \eqref{eq:shift-iso}, $m_n=s^{-1}\,b_n\,s^{\ot n}$, and $sR_{1;\omega}=s\,N_\omega\,s^{-1}$ has degree $0$. Since $R_r=0$ for $r\geq2$, the only term of \eqref{eq:MC-homotopy} which survives is $r_1=\cdots=r_p=1$, hence $p=n$. Fix $0\leq t\leq n$.
The inner brace $b_n\{sR_1,\ldots,sR_1\}$, with $n-t$ copies of $sR_1$, is $\sum_{|S|=n-t}b_n\circ(sR_1)^S$, where $(sR_1)^S=(sR_1)^S_{\alpha_1}\ot\cdots\ot(sR_1)^S_{\alpha_n}$, $(sR_1)^S_{\alpha_i}=sR_{1;\alpha_i}$ if $i\in S$ and $(sR_1)^S_{\alpha_i}=\id$ otherwise; no sign appears since $sR_1$ has degree $0$. Each of the $t$ outer braces is a composition $sR_{1;\omega}\circ-$ with $\omega=\alpha_1\cdots\alpha_n$, since $sR_1$ has a single input.
Hence the term indexed by $t$ is $(-1)^t\sum_{|S|=n-t}(sR_{1;\omega})^t\circ b_n\circ(sR_1)^S$. Since $(sR_{1;\omega})^k=s\,N_\omega^k\,s^{-1}$ and $(sR_1)^S\circ s^{\ot n}=s^{\ot n}\circ\big(N^S_{\alpha_1}\ot\cdots\ot N^S_{\alpha_n}\big)$, all the maps involved having degree $0$, composing \eqref{eq:MC-homotopy} with $s^{-1}$ on the left and $s^{\ot n}$ on the right, which are invertible, we obtain the equivalent identity
\[\sum_{S\subseteq\{1,\ldots,n\}}\big(-N_{\alpha_1\cdots\alpha_n}\big)^{n-|S|}\circ{m}_n\circ\big(N^S_{\alpha_1}\ot\cdots\ot N^S_{\alpha_n}\big)=0,\]
which is \eqref{eq:strict}. For $n=1$, it is $m_1\circ N_\alpha-N_\alpha\circ{m}_1=0$, that is, \eqref{eq:h1}.
\end{proof}

\subsection{\texorpdfstring{$\Omega$-$A_\infty$}{Omega-A-infinity}-algebras and strict homotopy Nijenhuis family algebras}
As in \cite[Section~6.2]{Ddas}, we introduce the family version of $A_\infty$-algebras.

\begin{defn}\label{def:OmegaAinfty}
An \emph{$\Omega$-$A_\infty$-algebra} is a pair $(A,\{\nu_n\}_{n\geq1})$ consisting of a graded vector space $A$ and collections $\nu_n=\{\nu_{n;\alpha_1,\ldots,\alpha_n}\}\in\Hom_\Omega(A^{\ot n},A)$ of linear maps of degree $n-2$ satisfying the identities \eqref{eq:Stasheff} in $\mathrm{End}^\Omega_A$, that is, for all $n\geq1$ and $\alpha_1,\ldots,\alpha_n\in\Omega$,
\[\sum_{\substack{i+j+k=n\\ i,k\geq0,\ j\geq1}}(-1)^{i+jk}\,\nu_{i+1+k;\,\alpha_1,\ldots,\alpha_i,\gamma,\alpha_{i+j+1},\ldots,\alpha_n}\circ\big(\id^{\ot i}\ot\nu_{j;\alpha_{i+1},\ldots,\alpha_{i+j}}\ot\id^{\ot k}\big)=0,\]
where $\gamma=\alpha_{i+1}\cdots\alpha_{i+j}$.
\end{defn}

For $\Omega$ trivial, this is an $A_\infty$-algebra. The following result is the analogue of \cite[Proposition~6.8]{Ddas}.

\begin{prop}\label{prop:OmegaAinfty}
Let $A$ be a graded vector space and $\nu_n=\{\nu_{n;\alpha_1,\ldots,\alpha_n}\}\in\Hom_\Omega(A^{\ot n},A)$, $n\geq1$, collections of linear maps of degree $n-2$. Define $\widetilde\nu_n$ on $A\ot\bk\Omega$ by
\[\widetilde\nu_n(a_1\ot\alpha_1,\ldots,a_n\ot\alpha_n):=\nu_{n;\alpha_1,\ldots,\alpha_n}(a_1,\ldots,a_n)\ot\alpha_1\cdots\alpha_n.\]
Then $(A\ot\bk\Omega,\{\widetilde\nu_n\})$ is an $A_\infty$-algebra if and only if $(A,\{\nu_n\})$ is an $\Omega$-$A_\infty$-algebra.
\end{prop}

\begin{proof}
With the notation of Lemma~\ref{lem:homog}, $\widetilde\nu_n=\iota(\nu_n)$. By Lemma~\ref{lem:transport}, applied to the graded space $A$, the left hand side of \eqref{eq:Stasheff} for $\{\widetilde\nu_n\}$ is the image under $\iota$ of the left hand side of \eqref{eq:Stasheff} for $\{\nu_n\}$ in $\mathrm{End}^\Omega_A$; no Koszul sign comes from $\bk\Omega$, which is concentrated in degree $0$. Since $\iota$ is injective, one vanishes if and only if the other does.
\end{proof}

We now show that strict homotopy Nijenhuis family algebras induce $\Omega$-$A_\infty$-algebras; this is the homotopy version of Proposition~\ref{prop:Omega-assoc}. For $\Omega$ trivial and $A$ concentrated in degree $0$, it reduces to \cite[Proposition~1.4]{song}.

\begin{thm}\label{thm:strict-Ainfty}
Let $(A,\{m_n\})$ be an $A_\infty$-algebra and $N=\{N_\omega\}_{\omega\in\Omega}$ a collection of linear maps of degree $0$ defining a strict homotopy Nijenhuis family algebra on $A$ (Proposition~\ref{prop:strict}). For $n\geq1$ and $\alpha_1,\ldots,\alpha_n\in\Omega$, put
\[\nu_{n;\alpha_1,\ldots,\alpha_n}:=\sum_{S\subsetneq\{1,\ldots,n\}}(-1)^{n-|S|+1}\,N_{\alpha_1\cdots\alpha_n}^{\,n-|S|-1}\circ m_n\circ\big(N^S_{\alpha_1}\ot\cdots\ot N^S_{\alpha_n}\big),\]
with $N^S_{\alpha_i}$ as in \eqref{eq:strict}. Then:
\begin{enumerate}
\item $\nu_1=m_1$ and $\nu_{2;\alpha,\beta}(a,b)=a\cdot_{N;\alpha,\,\beta}b=m_2(N_\alpha a,b)+m_2(a,N_\beta b)-N_{\alpha\beta}m_2(a,b)$;
\item $N_{\alpha_1\cdots\alpha_n}\circ\nu_{n;\alpha_1,\ldots,\alpha_n}=m_n\circ\big(N_{\alpha_1}\ot\cdots\ot N_{\alpha_n}\big)$ for all $n\geq1$;
\item $(A,\{\nu_n\}_{n\geq1})$ is an $\Omega$-$A_\infty$-algebra. Consequently, $A\ot\bk\Omega$, with
\[\widetilde\nu_n(a_1\ot\alpha_1,\ldots,a_n\ot\alpha_n)=\nu_{n;\alpha_1,\ldots,\alpha_n}(a_1,\ldots,a_n)\ot\alpha_1\cdots\alpha_n,\]
is an $A_\infty$-algebra, and $\widetilde N(a\ot\omega)=N_\omega(a)\ot\omega$ is a strict morphism of $A_\infty$-algebras from $(A\ot\bk\Omega,\{\widetilde\nu_n\})$ to $(A\ot\bk\Omega,\{\widetilde m_n\})$, where $\widetilde m_n(a_1\ot\alpha_1,\ldots,a_n\ot\alpha_n):=m_n(a_1,\ldots,a_n)\ot\alpha_1\cdots\alpha_n$.
\end{enumerate}
\end{thm}

\begin{proof}
For fixed $\alpha_1,\ldots,\alpha_n$, consider the linear operators on $\Hom(A^{\ot n},A)$ given by $\mathcal R_i(f):=f\circ(\id^{\ot i-1}\ot N_{\alpha_i}\ot\id^{\ot n-i})$ and $\mathcal L(f):=N_{\alpha_1\cdots\alpha_n}\circ f$. For $S=\{i_1<\cdots<i_k\}\subseteq\{1,\ldots,n\}$, put $\mathcal R_S:=\mathcal R_{i_1}\circ\cdots\circ \mathcal R_{i_k}$ (with $\mathcal R_\emptyset=\id$). Since $N$ has degree $0$, the operators $\mathcal R_1,\ldots,\mathcal R_n,\mathcal L$ commute with each other, and
\[\Phi^n_N(m_n)=\sum_{S\subseteq\{1,\ldots,n\}}(-\mathcal L)^{n-|S|}\mathcal R_S(m_n)=\big((\mathcal R_1-\mathcal L)\circ\cdots\circ(\mathcal R_n-\mathcal L)\big)(m_n),\]
\[\nu_n=\sum_{S\subsetneq\{1,\ldots,n\}}(-1)^{n-|S|+1}\mathcal L^{n-|S|-1}\mathcal R_S(m_n).\]
(1) For $n=1$, only $S=\emptyset$ occurs, and $\nu_1=m_1$. For $n=2$, the subsets $\{1\}$, $\{2\}$ and $\emptyset$ give the three terms of $\cdot_N$.

(2) We have
\[\mathcal L(\nu_n)=-\sum_{S\subsetneq\{1,\ldots,n\}}(-\mathcal L)^{n-|S|}\mathcal R_S(m_n)=\mathcal R_{\{1,\ldots,n\}}(m_n)-\Phi^n_N(m_n),\]
and $\Phi^n_N(m_n)=0$ by \eqref{eq:strict}. This is (2).

(3) Let $\lambda$ be a formal variable. All the maps extend $\bk(\!(\lambda^{-1})\!)$-linearly to $A(\!(\lambda^{-1})\!)$. Put $N^\lambda_\omega:=N_\omega+\lambda\,\id$; it has degree $0$ and is invertible, with inverse $(N^\lambda_\omega)^{-1}=\sum_{j\geq0}(-1)^jN_\omega^j\lambda^{-j-1}$. Replacing $N$ by $N^\lambda$ replaces $\mathcal R_i$ by $\mathcal R_i+\lambda$ and $\mathcal L$ by $\mathcal L+\lambda$, so
\[\Phi^n_{N^\lambda}(m_n)=\big((\mathcal R_1+\lambda)-(\mathcal L+\lambda)\big)\circ\cdots\circ\big((\mathcal R_n+\lambda)-(\mathcal L+\lambda)\big)(m_n)=\Phi^n_N(m_n)=0.\]
Let $\nu^\lambda_n$ be defined by the formula of the theorem with $N^\lambda$ instead of $N$. By (2) applied to $N^\lambda$, and since $N^\lambda_{\alpha_1\cdots\alpha_n}$ is invertible,
\begin{equation}\label{eq:nu-lambda}
\nu^\lambda_{n;\alpha_1,\ldots,\alpha_n}=\big(N^\lambda_{\alpha_1\cdots\alpha_n}\big)^{-1}\circ m_n\circ\big(N^\lambda_{\alpha_1}\ot\cdots\ot N^\lambda_{\alpha_n}\big).
\end{equation}
Fix $n\geq1$, $\alpha_1,\ldots,\alpha_n\in\Omega$ and $i+j+k=n$; put $\omega=\alpha_1\cdots\alpha_n$ and $\gamma=\alpha_{i+1}\cdots\alpha_{i+j}$. In $\mathrm{End}^\Omega_A$, the index of $\nu^\lambda_{i+1+k}$ is $(\alpha_1,\ldots,\alpha_i,\gamma,\alpha_{i+j+1},\ldots,\alpha_n)$, whose product is $\omega$. By \eqref{eq:nu-lambda},
\[\Big(\nu^\lambda_{i+1+k}\circ\big(\id^{\ot i}\ot\nu^\lambda_j\ot\id^{\ot k}\big)\Big)_{\alpha_1,\ldots,\alpha_n}=(N^\lambda_\omega)^{-1}\circ m_{i+1+k}\circ\big(\id^{\ot i}\ot m_j\ot\id^{\ot k}\big)\circ\big(N^\lambda_{\alpha_1}\ot\cdots\ot N^\lambda_{\alpha_n}\big),\]
since the factor $N^\lambda_\gamma$ applied by $\nu^\lambda_{i+1+k}$ to its $(i+1)$-th input cancels the factor $(N^\lambda_\gamma)^{-1}$ of $\nu^\lambda_{j;\alpha_{i+1},\ldots,\alpha_{i+j}}$, and since no Koszul sign appears when we interchange tensor products of maps of degree $0$ with $\id^{\ot i}\ot m_j\ot\id^{\ot k}$. Multiplying by $(-1)^{i+jk}$ and summing over $i+j+k=n$, we get
\begin{multline*}\sum_{i+j+k=n}(-1)^{i+jk}\Big(\nu^\lambda_{i+1+k}\circ\big(\id^{\ot i}\ot\nu^\lambda_j\ot\id^{\ot k}\big)\Big)_{\alpha_1,\ldots,\alpha_n}\\=(N^\lambda_\omega)^{-1}\circ\Big(\sum_{i+j+k=n}(-1)^{i+jk}m_{i+1+k}\circ\big(\id^{\ot i}\ot m_j\ot\id^{\ot k}\big)\Big)\circ\big(N^\lambda_{\alpha_1}\ot\cdots\ot N^\lambda_{\alpha_n}\big)=0,\end{multline*}
by \eqref{eq:Stasheff} for $\{m_n\}$. Now, by its defining formula,
\[\nu^\lambda_n=\sum_{S\subsetneq\{1,\ldots,n\}}(-1)^{n-|S|+1}(\mathcal L+\lambda)^{n-|S|-1}\mathcal R^\lambda_S(m_n),\]
where $\mathcal R^\lambda_S:=(\mathcal R_{i_1}+\lambda)\circ\cdots\circ(\mathcal R_{i_k}+\lambda)$ for $S=\{i_1<\cdots<i_k\}$. Hence $\nu^\lambda_n$ is a polynomial in $\lambda$ with coefficients in $\Hom(A^{\ot n},A)$, whose constant term is $\nu_n$. Thus the left hand side of the last identity is a polynomial $\sum_q\lambda^q c_q$ with $c_q\in\Hom(A^{\ot n},A)$, which vanishes on $A^{\ot n}\subseteq A(\!(\lambda^{-1})\!)^{\ot n}$; hence $c_q=0$ for all $q$. The constant term $c_0=0$ is \eqref{eq:Stasheff} for $\{\nu_n\}$ in $\mathrm{End}^\Omega_A$.

Finally, by Proposition~\ref{prop:OmegaAinfty} applied to $\{\nu_n\}$ and to the constant collections $\{m_n\}$, $(A\ot\bk\Omega,\{\widetilde\nu_n\})$ and $(A\ot\bk\Omega,\{\widetilde m_n\})$ are $A_\infty$-algebras, and by (2),
\begin{align*}
\widetilde N\big(\widetilde\nu_n(a_1\ot\alpha_1,\ldots,a_n\ot\alpha_n)\big)&=N_{\alpha_1\cdots\alpha_n}\big(\nu_{n;\alpha_1,\ldots,\alpha_n}(a_1,\ldots,a_n)\big)\ot\alpha_1\cdots\alpha_n\\
&=m_n\big(N_{\alpha_1}(a_1),\ldots,N_{\alpha_n}(a_n)\big)\ot\alpha_1\cdots\alpha_n\\&=\widetilde m_n\big(\widetilde N(a_1\ot\alpha_1),\ldots,\widetilde N(a_n\ot\alpha_n)\big),
\end{align*}
with no sign since $\widetilde N$ has degree $0$. Hence $\widetilde N$ is a strict morphism of $A_\infty$-algebras.
\end{proof}

\begin{remark}\label{rem:nonstrict}
For a non-strict homotopy Nijenhuis family algebra, \eqref{eq:MC-homotopy} also involves the components $N_n$, $n\geq2$, so the products $\nu_n$ would have to be corrected by terms involving them, and the conjugation by $N_\omega+\lambda\,\id$ no longer applies. The homotopy transfer theorem only gives structures on homotopy equivalent complexes, not on $A\ot\bk\Omega$ itself; we leave this case open.
\end{remark}

\section*{Declarations}

{\leftskip=1.5em
\noindent\textbf{Funding.} The author did not receive support from any organization for the submitted work.

\smallskip
\noindent\textbf{Conflict of interest.} The author declares that he has no conflict of interest.\par}

\end{document}